\documentclass[11pt]{amsart}
\usepackage{latexsym,amssymb,amsmath}
\usepackage{amssymb,amsthm,upref,amscd}
\usepackage{color}
\usepackage[T1]{fontenc}
\usepackage[titletoc]{appendix}
\usepackage{srcltx}
\usepackage[utf8]{inputenc}
\usepackage{graphicx,color,verbatim}
\usepackage{amscd,amsmath,amsfonts,amssymb}
\usepackage{mathrsfs}
\usepackage{hyperref}
\usepackage{color,enumitem,graphicx}
 \usepackage[]{xcolor}

\newtheorem{thm}{Theorem}[section]
\newtheorem{lem}[thm]{Lemma}
\newtheorem{cor}[thm]{Corollary}
\newtheorem{Prop}[thm]{Proposition}

\newtheorem{Rem}[thm]{Remark}

\numberwithin{equation}{section}

\newcommand{\R}{\mathbb R}

\title[Schr\"odinger--Newton equations in dimension two]{Schr\"odinger--Newton equations in dimension two via a Pohozaev--Trudinger log-weighted inequality}

\author[D.~Cassani]{Daniele Cassani$^\text{1}$}

\author[C.~Tarsi]{Cristina Tarsi$^\text{2}$}
\address[D.~Cassani]{\newline\indent Dip. di Scienza e Alta Tecnologia
	\newline\indent
	Universit\`{a} degli Studi dell'Insubria
	\newline\indent and
	\newline\indent RISM--Riemann International School of Mathematics
	\newline\indent Villa Toeplitz, Via G.B. Vico, 46 -- 21100 Varese}
    \email{\href{mailto:daniele.cassani@uninsubria.it}{daniele.cassani@uninsubria.it}}

\address[C.~Tarsi]{\newline\indent Dipartimento di Matematica
	\newline\indent
	Universit\`a degli Studi di Milano
	\newline\indent Via C. Saldini, 50 -- 20133 Milano 
	}
\email{\href{mailto:cristina.tarsi@unimi.it}{cristina.tarsi@unimi.it}}

\thanks{(1) Corresponding author: daniele.cassani@uninsubria.it}

\date{\today}

\begin{document}

\begin{abstract} We study the following Choquard type equation in the whole plane 
\begin{equation*}
(C)\quad -\Delta u+V(x)u=(I_2\ast F(x,u))f(x,u),\quad x\in\R^2
\end{equation*}
where $I_2$ is the Newton logarithmic kernel, $V$ is a bounded Schr\"odinger potential and the nonlinearity $f(x,u)$, whose primitive in $u$ vanishing at zero is $F(x,u)$, exhibits the highest possible growth which is of exponential type. The competition between the logarithmic kernel and the exponential nonlinearity demands for new tools. A proper function space setting is provided by a new weighted version of the Pohozaev--Trudinger inequality which enables us to prove the existence of variational, in particular finite energy solutions to $(C)$. 
\end{abstract}

\maketitle


\section{Introduction} 
\noindent Consider the following class of nonlocal equations 
\begin{equation}\label{qq1}
-\Delta u+V(x)u=(I_\alpha\ast F(u))f(u),\quad x\in\R^N
\end{equation}
where $V\geq 0$ is the external Schr\"odinger potential, $F$ is the primitive function of the nonlinearity $f$ vanishing at zero, the kernel $I_\alpha$ is defined for $x\in\R^N\setminus\{0\}$, $N\geq 2$ by 
\begin{equation*}
I_\alpha(x):=
\begin{cases}
\displaystyle\frac{\frac{\Gamma((N-\alpha)/2)}{\Gamma(\alpha/2)\pi^{N/2}2^\alpha}}{|x|^{N-\alpha}},\quad 0<\alpha<N, &\\
&\\
\displaystyle \frac{1}{2^{N-1}\pi^{\frac{N}{2}}\Gamma(\frac{N}{2})}\log\frac{1}{|x|},\quad \alpha=N,
\end{cases}
\end{equation*}
where $\Gamma(\cdot)$ denotes Euler's Gamma function. Notice that passing from $\alpha<N$ to the limiting case $\alpha=N$ the kernel is no longer of one sign and does not decay at infinity which sets the problem in a quite different framework. By introducing the function $\phi:=I_\alpha\ast F(u)$ one has that \eqref{qq1} is equivalent to the following system 
\begin{equation}\label{ss1}
\begin{cases}
-\Delta u+V(x)u=\phi f(u),\quad &\\
& x\in\R^N\\
-\Delta^{\frac{\alpha}{2}} \phi=F(u)\ ,
\end{cases}
\end{equation}
which in the case $\alpha=2$ it turns out to be the so-called Schr\"odinger--Poisson system which has an Hamiltonian structure and which turns out to be relevant in applications, see \cite{Benci} and references therein. An extensive literature has been devoted to the higher dimensional case $N\geq 3$ and we refer to \cite{MV,CVZ,CZ,Du,W} for an up to date, though non exhaustive bibliography. 
On the contrary, just a few results are available in the planar case. Existence results in the case $\alpha<N=2$ have been proved in \cite{BV} in the case of power-like nonlinearities and in \cite{ACTY} in the case of exponential growth.  However, in dimension two the equivalence between \eqref{qq1} and \eqref{ss1} in the Schr\"odinger--Poisson case $\alpha=2$, carries over as long as the logarithmic kernel is taken into account. Existence and qualitative properties of solutions in the case of logarithmic kernel have been obtained in \cite{BCV,CW} for power-like nonlinearities. On the other hand, the polynomial growth somehow downplays the main feature of dimension two which allows finite energy solutions to have arbitrary polynomial growth up to the exponential. 

\noindent The main purpose of this paper is to tackle the problem in which one has the logarithmic kernel and the exponential growth, namely the limiting case $\alpha=N=2$ which is in turn the Schr\"odinger-Poisson case. As we are going to see, the main difficulty arises in the competing presence of a too loose singular behavior of the logarithmic kernel compared with the exponential growth of the nonlinearity within the convolution, for which the problem demands for a proper function space setting. Here we develop a suitable framework in which we can prove the existence of mountain pass solutions. Let us finally mention that nonlinear terms with exponential growth outside the convolution, which cast the problem in a quite different context, have been recently considered in \cite{AF}.  

\medskip

\noindent We will focus on the following non-autonomous problem 
\begin{equation}\label{P}
\begin{cases}
\displaystyle -\Delta u +V(x) u  =\frac 1{2\pi}\Big(\log \frac 1{|x|}\ast F(x,u)\Big)f(x,u), \,\,\ \mbox{in} \,\,\, \mathbb{R}^{2} \\
u \in H^{1}(\mathbb{R}^{2}),\, u>0\ .
\end{cases}
\end{equation}
On the Schr\"odinger potential $V$ we make the following assumptions:
\begin{itemize}
	\item[$(V_1)$] $V(x)\geq V_0>0$ in $\R^2$ for some $V_0>0$;
	
	\item[$(V_2)$] $V(x)$ is a 1-periodic continuous function\ .
\end{itemize}

\noindent With a slight abuse of notation, we assume the nonlinearity $f(x,s)=c(x)f(s)$ where $c(x)$ is a strictly positive, $1$-periodic continuous function and $f(s)$ a differentiable function whose primitive vanishing at zero is $F(s)$ and such that:
\begin{itemize}
	\item[$(f_1)$]$\displaystyle f(s)\geq 0$ for all $s \geq 0$, $\displaystyle f(s) \leq Cs^{p} e^{4\pi s^2}$ as $s\to +\infty$   for some $p>0$ and $f(s)\asymp s^{q-1}$ for some $q\geq 2$ as $s\to 0$;  \\
	
	\item[$(f_2)$]  $\exists C>\delta>0$ such that $\delta \leq \frac{F(s) f'(s)}{f^2(s)}\leq C \quad  \forall \, 	s> 0$;\\ 
		
	\item[$(f_3)$]  $\displaystyle \lim_{s\to +\infty}\frac{F(s) f'(s)}{f^2(s)}= 1$, or equivalently 
	$\displaystyle \lim_{s\to +\infty}\frac{d}{d s}\frac{F(s)}{f(s)}=0$;\\
	
	\item[$(f_4)$] $\displaystyle
	\lim_{s\to +\infty} \frac{s^3f(s)F(s)}{e^{8\pi s^2}}\geq\beta> \mathcal V$, where $\mathcal V$ will be explicitly given in Section \ref{MP_geo}. 	
	
\end{itemize}
\noindent Since we look for positive solutions, we may also assume $f(s)=0$ for $s\leq 0$.

\noindent Let us make a few comments on our assumptions:
\begin{itemize}	
\item $(f_1)$ gives the following 
	\begin{equation}\label{estF}
	0\leq F(s)\leq C\cdot\left\{
	\begin{array}{ll}
	s^{q}, & s\leq s_0\\
	s^{p-1}e^{4\pi s^2}, & s>s_0
	\end{array}
	\right.  \quad \hbox{ for some } s_0>1
	\end{equation}
	observing that $\int_{s_0}^s t^pe^{4\pi t^2}dt\leq s^{p-1}\int_1^s te^{4\pi t^2}dt$ for any $s>1$;

\item $(f_2)$ implies $f(s)$ is monotone increasing in $s$, so that $F(s)=\int_0^sf(\tau)d\tau \leq sf(s)$. Hence, the quantity $\frac{F(x,s)}{f(x,s)}=\frac{c(x)F(s)}{c(x)f(s)}=\frac{F(s)}{f(s)}$ is well defined and vanishes only at $s=0$. Furthermore, 
	\begin{equation}\label{F/f}
\frac{\partial}{\partial s}\left(\frac{F(x,s)}{f(x,s)}\right)=\frac{d}{d s}\left(\frac{F(s)}{f(s)}\right)=\frac{f^2(s)-F(s)f'(,s)}{f^2(s)}\leq 1 -\delta
		\end{equation}
which implies $F(x,s)\leq (1-\delta) sf(x,s)$
and this  improves the previous Ambrosetti-Rabinowitz condition $F(x,s)\leq sf(x,s)$;

\item $(f_3)$ yields a fine estimate from below on the quotient  $\displaystyle\frac{Ff'}{f^2}$, as  $s\to +\infty$. 

\noindent Indeed, for any $\varepsilon>0$ there exists $s_\varepsilon>0$ such that:
\begin{equation}\label{estF/f}
\frac{Ff'}{f^2}(s)\geq \left\{
\begin{array}{ll}
 \delta s, & s\leq s_\varepsilon\\
(1-\varepsilon)s, & s>s_\varepsilon \ ;
\end{array}
\right. 
\end{equation}

\item $(f_4)$ is in the spirit of the de Figueiredo--Miyagaki--Ruf condition \cite{dFMR} which in dimension two turns out to be a key ingredient in order to prove compactness. Loosely speaking, it plays the role of the upper bound of the energy in terms of the Sobolev constant $(1/N)S^{N/2}$ in higher dimensions. The role of condition $(f_4)$ will be detailed in Section \ref{MP_geo};

\item Functions $F_i(s)$ satisfying our set of assumptions are for instance given by:
\begin{equation*}
F_1(s)= e^{4\pi s^2}-1; \ \ F_2(s)=s^pe^{4\pi s^2},\ \forall \, p \geq 2;\ \ F_3(s)= \left\{
\begin{array}{ll}
s^q, & s\leq s_0\\
cs^{p}e^{4\pi s^2}, & s>s_0
\end{array}, \ \forall\, q\geq 2,\,  p> 1\ .
\right.
\end{equation*}
\end{itemize}

	\noindent Accordingly to our assumptions on the nonlinearity, we will distinguish two cases, namely when $q=2$ and $q>2$. For the sake of clearness, we will state our main results in the case $q=2$, whence the general case when $q>2$ will be covered since Section \ref{log_p_t} and thereafter. Consider the following weighted Sobolev space $H_{w_0}^1(\R^2)$ which is the completion of smooth compactly supported functions with respect to the norm
\[
\|u\|_{w_0}^2=\|\nabla u\|_2^2+\|u\|_{L^2(w_0dx)}^2=\int_{\R^2}|\nabla u|^2dx+\int_{\R^2}u^2\log(e+|x|)dx\ .
\]

\begin{thm}\label{thm_A1}
The weighted Sobolev space $H_{w_0}^1(\R^2)$ embeds into the weighted Orlicz space $L_{\phi}(\R^2, \log(e+|x|)dx)$ where $\phi$ is the $n$-function $\phi(t)=e^{t^2}-1$. More precisely, we have 
\begin{equation}\label{wT}
\int_{\R^2}\left(e^{\alpha u^2}-1\right)\log(e+|x|)dx<\infty
\end{equation}
for any $u\in H^1_{w_0}(\R^2)$ and any $\alpha >0$.
Moreover, the following uniform bound holds 
\begin{equation}\label{wM}
\sup_{\|u\|_{w_0}^2\leq 1} \int_{\R^2}\left(e^{2\pi u^2}-1\right)\log(e+|x|)dx <+\infty\ .
\end{equation}
\end{thm}

\begin{thm}\label{thm1}
Suppose the nonlinearity $f$ satisfies $(f_1)$--$(f_4)$ and that the potential $V$ enjoys $(V_1)$--$(V_2)$. Then, problem \eqref{P} possesses a nontrivial finite energy solution (which in the case $q=2$ belongs to $H^1_{w_0}(\R^2)$).
\end{thm}
\begin{Rem} In the more general case $q>2$, the solution of Theorem \ref{thm1} has finte energy in the space $H^1L^q_{w}(\R^2)$ introduced in Section \ref{log_p_t}.
\end{Rem}

\medskip

\subsection*{Overview} Equation \eqref{qq1} has a long history, heritage of the early studies on Polarons iniziated by Fr\"ohlich \cite{fro} and then has been revealed a good model also in completely different contexts such as plasma physics \cite{lieb} and quantum gravity \cite{penrose}. We refer the interested reader to the survey \cite{MV} and references therein for more on Physical aspects of the problem. 

\noindent Formally, the energy associated to problem \eqref{P} is given by 
$$I_V(u)=\frac12\int_{\R^2}|\nabla u|^2+ V(x) u^2dx- \frac1{4\pi}\int_{\R^2}\Big[\log \frac{1}{|x|}\ast F(x,u)\Big]F(x,u)dx \ . $$
\noindent In order to have the energy well defined in presence of a logarithmic kernel, the authors in \cite{CW, Stubbe}, restrict the space $H^1$ introducing a further constraint, eventually setting the problem in an intersection space in which the energy turns out to be well defined by the Hardy--Littlewood--Sobolev inequality. Our approach here is different, from one side we look for a proper function space setting in which such a natural constraint turns out to be automatically satisfied and on the other side, we wonder if this can be done by allowing the nonlinearity to exhibit exponential growth which is what we expect in dimension two, since the seminal work of Pohozaev \cite{P} and Trudinger \cite{NT}. Indeed, we prove that the Sobolev space $H^1$ with a logarithmic weight on the $L^2$ mass term of the norm gives the proper function space setting in which the energy in well defined up to the natural exponential growth in the nonlinearity. Our argument throws light on the fact that, roughly speaking, as concentration phenomena in the Moser functional are controlled by the $L^2$ norm of the gradient whereas vanishing phenomena are controlled by the $L^2$ norm, here we prove that a suitable logarithmic weight in the $L^2$ component of the $H^1$ norm is enough to obtain a functional inequality which at the end yields a natural function space framework where to set up the problem. We think this result is of independent interest and that could be useful elsewhere.  As pointed out also in \cite{AF,CW} an extra difficulty is given here by the lack of invariance by translations of the energy which forces to prove a priori bounds of eventually vanishing Palais--Smale sequences. Our method seems to be more natural also in this respect, as starting from any PS sequence we can prove the existence of a weak $H^1$-limit with no need to establish a priori bounds. 

\medskip

\noindent For convenience of the reader, some preliminary material is recalled in Section \ref{prelim}. In Section \ref{log_p_t} we establish the fundamental embedding inequality which will provide the function space framework of Section \ref{functional_framework} and that will be used throughout the paper. Section \ref{MP_geo} is devoted to show the underlying mountain pass geometry for the energy functional and to prove mountain pass level estimates which in this case is a delicate matter. In Section \ref{PS_sec} we prove compactness results by carefully analyzing the behavior of PS sequences and finally, we conclude in Section \ref{final_sec} the proof of Theorem \ref{thm1}.

\section{Preliminaries}\label{prelim}
\noindent In this section we recall some well known results which will be used in the sequel.

\noindent Let  $H^1_0(\Omega)$ be the classical Sobolev space, completion of smooth compactly supported functions with respect to the Dirichlet norm $\|\nabla \cdot\|_2$, when $\Omega$ is a bounded subset of $\R^N$, and with respect to the complete Sobolev  norm $(\|\nabla \cdot\|^2_2+\|\cdot \|_2^2)^{1/2}$, when the domain is unbounded and in particular for $\Omega =\R^N$. 

\noindent If $N\geq 3$,  the  classical Sobolev embedding theorem reads as follows 
\begin{equation}\label{Si}
H^1_0(\Omega)\hookrightarrow L^{2^*}(\Omega) \ , \ \text{ namely }\ \|u\|_{2^*}\leq \frac{1}{S}\, \|\nabla u\|_2\ ,
\end{equation}
where $2^*:=\frac{2N}{N-2}$ is the critical Sobolev exponent and the constant $S$ in \eqref{Si} is the best possible \cite{T}. 

\noindent When $N=2$ is the so-called Sobolev limiting case. One has the embedding $H^1_0(\Omega)\hookrightarrow L^p(\Omega)$ for all $1\leq p < \infty$ (see also \cite{CTZ} for related best constants estimates), though $H^1_0(\Omega) \not \subset L^{\infty}(\Omega)$.
The maximal degree of summability for functions in $H^1_0(\Omega)$ was established independently by Poho\v{z}aev \cite{P} and Trudinger \cite{NT} (see also \cite{Yudovich}) and is of exponential type, in a suitable Orlic\v{z} class of functions, namely  
\begin{equation}\label{orlicz}
 u \in H^1_0(\Omega) \ \Longrightarrow \ \displaystyle\int_\Omega (e^{\alpha|u|^2}-1)\,dx<\infty
, \quad \forall \alpha>0\ .
\end{equation}
Starting from the seminal work of J.~Moser \cite{M} in which a sharp version of \eqref{orlicz} is established, the Pohozaev--Trudinger embedding has been further developed during the last fifty years, in particular the first extension of \eqref{orlicz} to unbounded domains appears in \cite{Cao} for functions with bounded Sobolev's norm in the following form 
\begin{equation}\label{caoineq}
\sup_{\|\nabla u\|_2\leq 1,\, \|u\|_2\leq M}\int_{\R^2}\left(e^{\alpha u^2}-1\right)dx\leq C(\alpha)\|u\|_2<\infty\,\text{ if } \,\alpha<4\pi
\end{equation}
Thereafter, several sharp versions have been proved and extensions in many directions for which we refer to \cite{R,CST}. In particular, the borderline case in which $\alpha=4\pi$ remained uncovered until Ruf in \cite{R} established the following inequality which is sharp in the sense of Moser \cite{M} (so-called Trudinger-Moser type inqualities):
\begin{equation}\label{Ri}
\sup_{\|\nabla u\|_2^2+\:\|u\|_2^2\leq 1}\int_{\mathbb R^2}\left(e^{\alpha
u^2}-1\right)~dx\leq \widetilde{C}(\alpha)<\infty  \iff  \alpha\leq 4\pi\ .
\end{equation}
\begin{Rem}\label{Ruf_rem}
Note that in Ruf's inequality \eqref{Ri} the constraint is defined through the complete Sobolev norm  $\|\nabla \cdot\|_2^2+\|\cdot\|_2^2$. As one may realize by Cao's result, a closer inspection of the proof reveals that Ruf's inequality still holds, at least in the subcritical case $\alpha <4\pi$, replacing the $L^2$ norm with any weighted $L^2$ norm, provided the weight is bounded and also bounded away from the origin. 
\end{Rem}
\begin{Rem}
As an application of  \eqref{orlicz}, consider the following functional
\begin{equation*}
H^1(\R^2)\ni u\longmapsto  \int_{\R^2}F(x,u)\, dx
\end{equation*}
which is continuous on $H^1(\R^2)$, a consequence of $(f_1)$, \eqref{estF} and Holder's inequality. Indeed, note first that for any $t, s >0$
$$
|F(t)-F(s)|=\left|\int_s^tf(\tau) d\tau\right|\leq C \left|\int_s^t \left(\tau^{q-1}+\tau^{p-1}e^{4\pi \tau^2}\right)d\tau \right|
$$
so that, if $u_n\to u$ in $H^1(\R^2)$, as $n\to \infty$, then
\begin{multline*}
\int_{\R^2}|F(x,u_n)-F(x,u)|dx\leq  C\|u_n-u\|_2^2+ C
\int_{\R^2}|e^{5\pi u_n^2}-e^{5\pi u^2}|dx\\ 
\leq {\text{o}}(1)+ C\int_{\R^2}\left(e^{5\pi u^2}-1\right)|e^{5\pi (u_n^2-u^2)}-1|dx +C\int_{\R^2}\left|e^{5\pi (u_n^2-u^2)}-1\right|dx\longrightarrow 0
\end{multline*}
and the same holds  for the functional $u\longmapsto  \int_{\R^2}uf(x,u)\, dx$.

\end{Rem}

\noindent The main feature of the equation \eqref{qq1} is the nonlocal term defined through a convolution product. This turns out to be well defined in view of the following Hardy--Littlewood--Sobolev inequalities, which we state in  $\R^N$ for any $N\geq 1$, see \cite{LL} and also \cite{BS} for the interpolation spaces approach.

\begin{Prop}[HLS inequality]\label{HLS}
	Let $s, r>1$ and $0<\mu<N$ with $1/s+\mu/N+1/r=2$, $f\in
	L^s(\R^N)$ and $g\in L^r(\R^N)$. There exists a constant
	$C(s,N,\mu,r)$, independent of $f,h$, such that
	$$
	\int_{\R^N}[\frac{1}{|x|^{\mu}}\ast f(x)]g(x)\leq
	C(s,N,\mu,r) \|f\|_s\|g\|_r.
	$$
	\end{Prop}
\begin{Rem}	
Note that the Sobolev inequality \eqref{Si} is equivalent, by duality, to a special case of the HLS inequality \eqref{HLS} (see \cite{B}). Actually, take $\mu=N-2$ and $s=r=2^*$: then \eqref{HLS} says that the inclusion $L^{\frac{2N}{N+2}}(\R^N)\hookrightarrow H^{-1}(\R^N)$ is continuous, and so, by duality, its counterpart, $H^1(\R^N)\hookrightarrow L^{2^*}(\R^N)$. 
\end{Rem}
\noindent By exploiting a limiting procedure as $\mu \to 0$, one can prove the so-called logarithmic Hardy--Littlewood--Sobolev inequality, whose main feature is the presence of a sign-changing logarithmic kernel, see \cite{B,CL,dPDF}.

\begin{Prop}[Logarithmic HLS
	inequality]\label{LHLS}Let $f,g$ be two nonnegative
	functions belonging to $L\ln L(\R^N)$, such that $\int f \log(1+|x|)<\infty, \int g \log(1+|x|)<\infty $ and $\|f\|_1=\|g\|_1=1$. There exists a constant $C_N$,
	independent of $f,g$, such that
\begin{equation}\label{LHLSeq}
	2N\int_{\R^N}[\log \frac{1}{|x|}\ast f(x)]g(x)\leq 
	C_N+\int_{\mathbb R^N}f\log f
	dx+\int_{\mathbb R^N}g\log g dx\ .
	\end{equation}
\end{Prop}
\begin{Rem}
Let us stress the feature of the log kernel, which has variable sign, and it is unbounded both in $0$ and at $+\infty$. This justifies the presence of the additional condition, $f, g \in L\ln L(\R^N)$ in order to have the inequality \eqref{LHLSeq} well defined, and in particular that no cancellation of infinities occurs. However, this does not imply the boundedness of $\int f\log f, \int g\log g$, but only of the positive parts $\int f\log_+f$, $\int g\log_+g$. The further weight conditions $\int f\log(1+|x|),\int g\log(1+|x|)<\infty$  make both the two sides of inequality \eqref{LHLSeq} finite. 
\end{Rem}

\section{A log-mass weighted Pohozaev--Trudinger type inequality}\label{log_p_t}
\noindent This Section is devoted  to prove a Pohozaev--Trudinger type inequality in the whole plane $\R^2$, with a logarithmic weight which appears only in the mass component of the energy. Here, the logarithmic weight plays a role only as $|x|\to +\infty$, for which we consider as prototype weight $w_0=\log(e+|x|)$.  
\noindent On the other hand, it is well known from \cite{R,ISHI,CST}, how the growth near zero is a key ingredient in proving Pohozaev--Trudinger type inequalities on unbounded domains, since it is strictly related to vanishing phenomena.  Here we aim at proving a fundamental inequality which will provide a suitable variational setting for \eqref{P}. Let us point out that the presence of an increasing weight prevents one to use rearrangement arguments.

\subsection{Proof of Theorem \ref{thm_A1}}
 Let us first perform a change of variables which enables one to pass from $H_{w_0}^1(\R^2)$ to functions in $H^1(\R^2)$ . Note that the inverse transformation does not turn out to be explicit and this is why we can not expect to prove directly our inequality. 

\noindent Let us use polar coordinates in $\R^2$:
$$
x=(x_1,x_2)=|x|(\cos \theta, \sin \theta), \ \ \hbox{ where } \ |x|=\sqrt{x^2_1+ x_2^2}
$$
We perform the change of variable
$$
y=(y_1, y_2)=|x|\sqrt{\log(e+|x|)}(\cos \theta, \sin \theta)
$$
which acts only on the radial part of any point in $\R^2$, equivalently
$$
T(|x|)=|y|, \ \ \frac y{|y|}=\frac x{|x|}, \ \ |y|=|x|\sqrt{\log(e+|x|)}
$$
In order to simplify the notation, set $r=|x|$ and $s=|y|$, so that the transformation  becomes $s=T(r)=r\sqrt{\log(e +r)}$. Note that 
$$
T'(r)=\frac{2\log(e+r)+\frac{r}{e+r}}{2\sqrt{\log(e+r)}}>0, \ \ T(0)=0, \ \lim_{r\to +\infty} T(r)=+\infty
$$
and thus $T$ is invertible on $\R^2$, though the inverse map is not explicitly known. \\
\noindent Define
\[
v(y):=u(x), \ \ \hbox{that is, } \ \ v(y)=u\left(T^{-1}(|y|)\cos \theta, T^{-1}(|y|)\sin \theta\right) 
\]
or, equivalently
\[ \ u(r\cos \theta, r\sin \theta)=v\left(T(r)\cos \theta, T(r)\sin \theta\right)\ .
\]
Then, by a direct calculation, if 
$$
w(r, \theta):=u(r\cos \theta, r \sin \theta) \  \ \widetilde w(s, \theta):=v\left(s\cos \theta, s \sin \theta \right), \ \ 
w(r, \theta)= \widetilde w\left( T(r), \theta \right)
$$
we have
\[
w_r(r, \theta)= \widetilde w_s\left( T(r), \theta \right)T'(r), \  w_{\theta}\left( T(r), \theta \right)=\widetilde w_{\theta}\left( T(r), \theta \right)
\]
so that
\begin{multline*}
\int_{\R^2}|\nabla v|^2dy_1dy_2=\int_{0}^{2\pi}\int_0^{+\infty}\left[\widetilde w_s^2+\frac{\widetilde w^2_\theta}{s^2}\right]sdsd\theta\\= \int_{0}^{2\pi}\int_0^{+\infty}\left[\widetilde w_s^2(T(r), \theta)+\frac{\widetilde w^2_\theta(T(r), \theta)}{T^2(r)}\right]T'(r)T(r)drd\theta\\
= \int_{0}^{2\pi}\int_0^{+\infty}\left[ w_r^2(r, \theta)\cdot \frac{1}{[T'(r)]^2}+\frac{ w^2_\theta(r, \theta)}{r^2}\cdot \frac{r^2}{T^2(r)}\right]T'(r)T(r)drd\theta\ .
\end{multline*}
Now,
\[
\frac{1}{[T'(r)]^2}=\frac{\log(e+r)}{\left[\log(e+r)+\frac{r}{2(e+r)}\right]^2}, \ \ \frac{r^2}{T^2(r)}=\frac{1}{\log(e+r)}
\]
so that
\[
\frac 13 \frac{r^2}{T^2(r)}< \frac{1}{[T'(r)]^2}< \frac{r^2}{T^2(r)}\ .
\]
Then,
\begin{multline*}
\frac 13 \int_{0}^{2\pi}\int_0^{+\infty}\left[ w_r^2+\frac{ w^2_\theta}{r^2}\right] \frac{r^2T'(r)}{T(r)}drd\theta\leq \int_{\R^2}|\nabla v|^2dy_1dy_2\\
\leq \int_{0}^{2\pi}\int_0^{+\infty}\left[ w_r^2+\frac{ w^2_\theta}{r^2}\right] \frac{r^2T'(r)}{T(r)}drd\theta\ .
\end{multline*}
Noting that
\[
\frac{r^2T'(r)}{T(r)}= r\left[1+ \frac r{2(e+r)\log(e+r)}\right] \Longrightarrow  r<\frac{r^2T'(r)}{T(r)}<2 r
\]
we eventually get
\begin{equation*}\label{Tgrad}
\frac 13 \int_{\R^2}|\nabla u|^2dx_1dx_2<\int_{\R^2}|\nabla v|^2dy_1dy_2<2\int_{\R^2}|\nabla u|^2dx_1dx_2
\end{equation*}
On the other hand,
\begin{multline*}
\int_{\R^2}v^2dy=\int_0^{2\pi}\int_0^{\infty} \widetilde w^2(s, \theta)sdsd\theta= \int_0^{2\pi}\int_0^{\infty} \widetilde w^2(T(r), \theta)T'(r)T(r)drd\theta\\=\int_0^{2\pi}\int_0^{\infty} w^2(r, \theta)T'(r)T(r)drd\theta\ .
\end{multline*}
Since
\begin{multline*}
T'(r)T(r)=r\left[\log(e+r)+\frac{r}{2(e+r)}\right]\\
=r\log(e+r)\left[1+\frac{r}{2(e+r)\log(e+r)}\right]=\frac{r^2T'(r)}{T(r)}\log(e+r)
\end{multline*}
we conclude that
\[
\int_{\R^2}v^2dy=\int_0^{2\pi}\int_0^{\infty} w^2\frac{r^2T'(r)}{T(r)}\log(e+r) drd\theta
\]
and, in turn
\[
\int_{\R^2}u^2\log(e+|x|)dx<\int_{\R^2}v^2dy<2\int_{\R^2}u^2\log(e+|x|)dx\ .
\]
Finally,
\begin{equation}
\frac 13\|u\|_{w_0}^2<\|v\|^2=\|\nabla v\|_2^2+\|v\|_2^2<2\|u\|_{w_0}^2\ .
\label{normuv}
\end{equation}
We have then proved that the map
\begin{eqnarray*}
 \mathfrak{T}: H^1_{w_0}(\mathbb R^2)&\to& H^1_0(\mathbb R^2)\\
 u&\mapsto& v
\end{eqnarray*}
is an invertible, continuous and with continuous inverse map. Then,
\begin{multline*}
 \int_{\R^2}\left(e^{ \alpha u^2}-1\right)\log(e+|x|)dx=\int_0^{2\pi}\int_0^{+\infty}\left(e^{ \alpha u^2(r\cos \theta, r\sin \theta)}-1\right)\log(e+r)rdrd\theta\\
 =\int_0^{2\pi}\int_0^{+\infty}\left(e^{ \alpha v^2(T(r)\cos \theta, T(r)\sin \theta)}-1\right)\frac{\log(e+r)r}{T'(r)T(r)}T'(r)T(r)drd\theta\\
  \leq \int_0^{2\pi}\int_0^{+\infty}\left(e^{ \alpha v^2(\rho\cos \theta, \rho\sin \theta)}-1\right)\rho d\rho d\theta
  =  \int_{\R^2}\left(e^{ \alpha v^2}-1\right)dx <+\infty
\end{multline*}
by  \cite{R} for any $\alpha>0$.
The uniform bound \eqref{wM} follows directly from  \eqref{normuv}, as for any $u\in H^1_{w}$ and $\alpha \leq 2\pi$ one has 
\begin{multline*}
 \int_{\R^2}\left(e^{ 2\pi u^2/\|u\|_{w_0}^2}-1\right)\log(e+|x|)dx\\
 =\int_0^{2\pi}\int_0^{+\infty}\left(e^{ 2\pi u^2(r\cos \theta, r\sin \theta)/\|u\|_{w_0}^2}-1\right)\log(e+r)rdrd\theta\\
 =\int_0^{2\pi}\int_0^{+\infty}\left(e^{ 2\pi v^2(T(r)\cos \theta, T(r)\sin \theta)/\|u\|_{w_0}^2}-1\right)\frac{\log(e+r)r}{T'(r)T(r)}T'(r)T(r)drd\theta\\
  \leq \int_0^{2\pi}\int_0^{+\infty}\left(e^{ 4\pi v^2(\rho\cos \theta, \rho\sin \theta)/\|v\|^2}-1\right)\rho d\rho d\theta
  =  \int_{\R^2}\left(e^{ 4\pi v^2/\|v\|^2}-1\right)dx <C(\alpha)
\end{multline*}
again by  \cite{R}.

\noindent A consequence of this embedding result is the continuity of the weighted Pohozaev--Trudinger functional on $H^1_{w_0}(\R^2)$, namely we have 
\begin{cor}\label{cor-wcontinuity}
For any $\alpha >0$ the functional
\begin{eqnarray*}
H^1_{w_0}(\R^2)& \longrightarrow & \R\\
u &\longmapsto & \int_{\R^2}\left(e^{\alpha u^2}-1\right)\log(e+|x|) dx
\end{eqnarray*}
is continuous.
\end{cor}
\begin{Rem}\label{remA3}
The value $2\pi$ in \eqref{wM} is not sharp and the problem of establishing a sharp version of \eqref{wM} in the spirit of Moser \cite{M} remains essentially open. 
\end{Rem}

\subsection{The case $q>2$} 

\noindent Let us now consider the more general case of a growth function $F(x,s)\asymp s^q$ as $s\to 0$ with $q>2$. In this case, the natural weighted Sobolev space turns out to be $H^1L^q_{w_0}(\R^2)$, defined as the completion of smooth compactly supported functions with respect to the norm 
\[
\|u\|_{w_0}^2=\|\nabla u\|_2^2+\|u\|_{L^q(w_0dx)}^{2/q}=\int_{\R^2}|\nabla u|^2dx+\int_{\R^2}|u|^q\log(e+|x|)dx\ .
\]
\begin{thm}\label{thm_Aq}
Let $F(x,s)$ sutisfy assumption $(f_1)$. Then, the space $H^1L^q_{w_0}(\R^2)$ embeds into the weighted Orlicz space $L_{F}(\R^2, \log(e+|x|)dx)$. More precisely, 
\begin{equation}\label{wTq}
\int_{\R^2}F(x,\alpha |u|)\log(e+|x|)dx<+\infty
\end{equation}
for any $u\in H^1L^q_{w_0}(\R^2)$ and any $\alpha >0$.

\noindent Furthermore, for any $\alpha \leq 1/\sqrt{q}$ one has 
\begin{equation}\label{wMq}
\sup_{\|u\|_{w_0}^2\leq 1} \int_{\R^2} F\left(x,\alpha |u|\right)\log(e+|x|)dx <+\infty\ .
\end{equation}
\end{thm}
 \begin{proof}The result will follow from the following estimate
 $$
 F(x,s)\leq C \begin{cases}
s^q, & \hbox{ if } s<1\\
 s^{p-1}e^{4\pi s^2}, & \hbox{ if } s\geq 1            
          \end{cases}
 $$
 Indeed, let $u\in H^1L^q_{w_0}$ be such that $\|u\|_{w_0}\leq 1$ and define
 $$
 v=\begin{cases}
           |u|^{q/2}, & |u|< 1\\
           |u|, & |u| >1
          \end{cases}\ . 
$$
Then
\begin{multline*}
\|v\|_{w_0}^2=\|\nabla v\|_2^2+\|v\|_{w_0}^{2}\leq \frac q2\int_{\R^2}|\nabla u|^2dx\\+\int_{|u|<1}|u|^q\log(e+|x|)dx+ \int_{|u|>1}|u|^2\log(e+|x|)dx\\
\leq \frac q2\int_{\R^2}|\nabla u|^2dx+\int_{\R^2}|u|^q\log(e+|x|)dx \leq \frac q2 \|u\|_{w_0}^2\leq \frac q2\ .
\end{multline*}
Hence
 \begin{multline*}
 \int_{\R^2}F(\alpha |u|)\log(e+|x|)dx\leq C\int_{\R^2}\alpha^q |u|^q\log(e+|x|)dx\\+C\int_{|u|>1}\alpha^{p-1}|u|^{p-1}e^{\alpha^2 4\pi u^2}\log(e+|x|)dx\\
 = C\alpha ^q \|u\|_{L^q(w_0dx)}^{2/q}+C\int_{|v|>1}\alpha^{p-1}|v|^{p-1}e^{\alpha^2 4\pi v^2}\log(e+|x|)dx\\
 \leq C_\alpha \|u\|_{L^q(w_0dx)}^{2/q}+C\alpha^{p-1}\left\{\int_{|v|>1}|v|^{\frac{\alpha^2q}{1-\alpha^2q}(p-1)}\right\}^{\frac{1-\alpha^2q}{\alpha^2q}}\\\cdot \, \left\{\int_{|v|>1}\left(e^{4\pi v^2/q}-1\right)\log(e+|x|)dx\right\}^{\alpha^2q}\\
 \leq C_{\alpha,q}+C_{\alpha,q,p}\left\{\int_{\R^2}\left(e^{2\pi v^2/\|v\|_{w_0}^2 }-1\right)\log(e+|x|)dx\right\}^{\alpha^2q}\leq C_{\alpha, p,q}
 \end{multline*}
  where we have applied Theorem \ref{thm_A1} and embeddings for weighted Sobolev spaces.
 \end{proof}
\noindent Clearly, a corollary similar to \ref{cor-wcontinuity} holds also in this case
\begin{cor}\label{cor-qwcontinuity}
	For any $\alpha >0$ the functional
	\begin{eqnarray*}
		H^1L^q_{w_0}(\R^2)& \longrightarrow & \R\\
		u &\longmapsto & \int_{\R^2}F(x,\alpha |u|)\log(e+|x|) dx
	\end{eqnarray*}
	is continuous, where $F$ is a function satisfying assumption $(f_1)$.
\end{cor}
\noindent  In analogy to our case, weighted Pohozev--Trudinger inequalities allowing increasing monomial weights have been proved in \cite{DL,nguyen}. Finally, let us mention that related inequalities in Sobolev spaces with respect to $log$-weighted measures can be found in \cite{CR}, where the authors consider a logarithmic weight in the full Sobolev norm.

\section{The functional framework}\label{functional_framework}

\noindent The energy functional related to \eqref{P} is given by
$$
\aligned I_{V}(u)=\frac12\|u\|^2_{V}-\mathfrak{F}(u).
\endaligned
$$
where
$$
\mathfrak{F}(u)=\frac1{4\pi}\int_{\R^2}\Big[\log \frac{1}{|x|}\ast
F(x,u)\Big]F(x,u)\,dx,
$$
and
$$
\|u\|_{V}:=\left(\int_{\R^2}|\nabla u|^2+ V(x) u^2dx\right)^{1/2}\ .
$$
Thanks to the assumptions on the potential $V$, $\|u\|_V$ is equivalent to the standard Sobolev norm. However, the energy functional $I_V$ is not well defined on $H^1(\R^2)$ due to the nonlocal term $\mathfrak{F}(u)$. In view of the logarithmic HLS inequality \eqref{LHLS},  Stubbe first proposed in the unpublished paper \cite{Stubbe} to set the problem in the intersection space adding the integrability condition for which the energy is well defined. Let $H^1L^q_{w}(\mathbb R^2)$ be the completion of smooth compactly supported functions with respect to the norm  
\begin{multline}\label{norm_Hw}
\|u\|_{1,q(w)}^2:=\|u\|_V^2+ \|u\|_{L^q(wdx)}^{2/q}\\
=\int_{\mathbb R^2}|\nabla u|^2dx+\int_{\mathbb R^2}u^2V(x)dx + \left\{\int_{\mathbb R^2}|u|^q\log (1+|x|))dx\right\}^{2/q}
\end{multline}
which in the case $q=2$ is induced by the scalar product 
\begin{equation*}
\left\langle u,v\right\rangle= \int_{\R^2}\nabla u \nabla v dx +\int_{\R^2} V(x)uvdx +\int_{\R^2} uv \log(1+|x|^2)dx\ .
\end{equation*}
Accordingly to \cite[Theorem 1.11]{OK}, $H^1L^q_{w}$ is a Banach space whose dual can be characterized as follows \cite[Theorem 14.9]{Simon}
$$
H^{-1}L^q_{w}(\mathbb R^2)=\left(H^1(\mathbb R^2)\cap L^q(\mathbb R^2,wdx)\right)'=H^{-1}(\mathbb R^2)+(L^q)'(\mathbb R^2,wdx)\lvert_{H^1L^q_{w}}\ .
$$
Next, the Orlicz type embedding established in Section \ref{log_p_t} will enable us to apply the logarithmic version of the HLS inequality in order to have the energy functional well defined on $H^1L^q_{w}$ and sufficiently smooth for variational purposes. 

\noindent In what follows we will use extensively the following elementary identity
\begin{equation*}
\log \frac 1r = \log \left(1+\frac 1r\right)-\log \left(1+r\right) \ .
\end{equation*}
Let us introduce the following bilinear forms: 
\begin{eqnarray*}
 (u,v)&\mapsto& \mathfrak B_1(u,v)=\int_{\R^2}\left[\log(1+|x|)\ast u\right]v dx= \int_{\R^2}\int_{\R^2}\log(1+|x-y|) u(x)v(y)dxdy;\\
(u,v)&\mapsto& \mathfrak B_2(u,v)=\int_{\R^2}\left[\log(1+\frac 1{|x|})\ast u\right]v dx= \int_{\R^2}\int_{\R^2}\log\left(1+\frac 1{|x-y|}\right) u(x)v(y)dxdy;\\
(u,v)&\mapsto& \mathfrak B_0(u,v)=\int_{\R^2}\left[\log \frac 1{|x|}\ast u\right]v dx=\mathfrak B_2(u,v)-\mathfrak B_1(u,v)\ .
\end{eqnarray*}
Since $\log(1+t)\leq t$ for any $t>0$, thanks to \eqref{HLS} with $\mu=1, N=2$ we have that $\mathfrak B_2(u,v)$ is well defined on $L^{4/3}(\R^2)\times L^{4/3}(\R^2)$, and
\begin{equation*}
|\mathfrak B_2(u,v)| \leq \int_{\R^2}\int_{\R^2} \frac 1{|x-y|}|u|(x)|v|(y)dxdy \leq C\|u\|_{4/3}\|v\|_{4/3}
\end{equation*}
Since $\log(1+|x-y|)\leq \log(1+|x|+|y|)\leq \log(1+|x|)+\log(1+|y|)$ we also have 
\begin{multline*}
|\mathfrak B_1(u,v)| \leq \int_{\R^2}\int_{\R^2} \left[\log(1+|x|)+\log(1+|y|)\right]|u|(x)|v|(y)dxdy \\
= \|u\|_{L^1(w)}\|v\|_{1}+\|u\|_{1}\|v\|_{L^1(w)}
\end{multline*}
where for simplicity we wriwe $\|\cdot\|_{L^1(w)}$ in place of $\|\cdot\|_{L^1(wdx)}$. Evaluating the bilinear forms $\mathfrak B_i(\cdot,\cdot)$ on the field $(F(x,u),F(x,u))$ we obtain the functionals:
\begin{eqnarray*}
&& \mathfrak F_1:H^1L^q,w\to [0, +\infty),\  \ \mathfrak F_1(u)=\mathfrak B_1(F(x,u),F(x,u))\\
&& \hspace*{5cm}= \int_{\R^2}\int_{\R^2}\log(1+|x-y|) F(x,u(x))F(y,u(y))dxdy;\\
&&\mathfrak F_2:H^1L^q_w\to [0, +\infty), \ \  \mathfrak F_2(u)=\mathfrak B_2(F(x,u),F(x,u))\\
&& \hspace*{5cm}= \int_{\R^2}\int_{\R^2}\log\left(1+\frac 1{|x-y|}\right) F(x,u(x))F(y,u(y))dxdy;\\
&& \mathfrak F:H^1\to [0, +\infty),  \mathfrak F(u)=\mathfrak B_0(F(x,u),F(x,u))=\mathfrak F_2(u)-\mathfrak F_1(u)\ .
\end{eqnarray*}
Reasoning as done above for the bilinear forms $\mathfrak B_i$, one has that $\mathfrak F_2(u)$ is well defined on $H^1L^q_w(\R^2)$ (actually, on the larger Sobolev space $H^1(\R^2))$ by assumption $(f_1)$ and recalling \eqref{orlicz}). Whereas for $\mathfrak F_1(u)$, observe that the quantity $\|F(x,u)\|_{L^1(wdx)}$ is always finite, for any $u\in H^1L^q_w$, by combining \eqref{estF} with Theorem \ref{thm_A1} and Theorem \ref{thm_Aq}.
\begin{Rem}
Note that for instance when $q=2$, the weight in the $L^2$ component of the norm is given by $V(x)+\log(1+|x|)$ for which we have $C_1[V(x)+\log(1+|x|)]\leq \log(e+|x|)\leq C_2[V(x)+\log(1+|x|)]$ for some positive constants $C_i$. Thus the norms involved turn out to be equivalent and this extends to $q>2$ as the norm $\|\cdot\|_V$ is equivalent to the Sobolev norm. 
\end{Rem}

\subsection{Regularity of the energy functional}
The main goal of this section is to prove regularity of the energy functional $I_{V}$. We have the following 
 \begin{lem}
	The functionals $\mathfrak F_1, \mathfrak F_2, \mathfrak F$ and $I_V$ are $\mathcal C^1$ on $H^1L^q_w(\R^2)$.
\end{lem}
\begin{proof}

\noindent Consider $\mathfrak F_1$ and let $\{u_n\}$ be a sequence in $H^1L^q_w(\R^2)$ converging to some $u$. Then
	\begin{multline*}
	|\mathfrak F_1(u_n)-\mathfrak F_1(u)|\\
	\leq\int_{\R^2}\int_{\R^2}\log(1+|x-y|) \left|F(x,u_n(x))F(y,u_n(y))-F(x,u(x))F(y,u(y))\right|dxdy\\
	\leq \int_{\R^2}\int_{\R^2}\log(1+|x-y|) F(x,u_n(x))\left|F(y,u_n(y))-F(y,u(y))\right|dxdy+\\
	+\int_{\R^2}\int_{\R^2}\log(1+|x-y|) \left|F(x,u_n(x))-F(x,u(x))\right|F(y,u(y))dxdy\\
	\leq \int_{\R^2}\int_{\R^2}\log(1+|x|) F(x,u_n(x))\left|F(y,u_n(y))-F(y,u(y))\right|dxdy+\\
	+\int_{\R^2}\int_{\R^2}\log(1+|y|) \left|F(y,u_n(y))-F(y,u(y))\right|F(x,u_n(x))dxdy\\
+	\int_{\R^2}\int_{\R^2}\log(1+|x|) F(y,u(y))\left|F(x,u_n(x))-F(x,u(x))\right|dxdy+\\
	+\int_{\R^2}\int_{\R^2}\log(1+|y|) \left|F(x,u_n(x))-F(x,u(x))\right|F(y,u(y))dxdy=I_1+I_2+I_3+I_4\ .
	\end{multline*}
	Now the terms $I_i$ tend to $0$, as $n\to\infty$ thanks to the continuity of the functional $\int_{\R^2}F(x,u)dx$ on $H^1$ and of the functional $\int_{\R^2}F(x,u)\log(1+|x|)dx$ on $H^1L^q_{w}$, as a consequence of \eqref{estF} and corollaries \ref{cor-wcontinuity}, \ref{cor-qwcontinuity}. For instance for $I_1$ we have 
	\begin{multline*}
	\int_{\R^2}\int_{\R^2}\log(1+|x|) F(x,u_n(x))\left|F(y,u_n(y))-F(y,u(y))\right|dxdy=\\
	=\left[\int_{\R^2}\log(1+|x|) F(x,u_n(x))dx\right]
	\cdot \left[\int_{\R^2}\left|F(y,u_n(y))-F(y,u(y))\right|dy\right]\\
	\leq \left[\int_{\R^2}\log(e+|x|) F(x,u(x))dx+{\text{o}}(1)\right]{\text{o}}(1)\ .
	\end{multline*}
	
\noindent For any $u\in H^1L^q_w$ the G\^ateaux derivative of $\mathfrak F_1$ at $u\in H^1L^q_w$ is given by
$$
	\mathfrak F_1'(u)[v]= 2\int_{\R^2}\int_{\R^2}\log(1+|x-y|) F(x,u(x))v(y)f(y,u(y))dxdy,\quad v\in H^{-1}L^q_{w}\ .
	$$
	Since
	\begin{multline*}
	|\mathfrak F_1'(u)v |\leq  2\int_{\R^2}\int_{\R^2}\log(1+|x|) F(x,u(x))v(y)f(y,u(y))dxdy\\
	+ 2\int_{\R^2}\int_{\R^2}\log(1+|y|) F(x,u(x))v(y)f(y,u(y))dxdy\\
	\leq  2\|v\|_2\left(\int_{\R^2} f^2(x,u)dx\right)^{1/2}\int_{\R^2} \log(1+|x|)F(x,u)dx + \\
	+2 \left(\int_{\R^2} \log(1+|y|)v^q dy\right)^{1/q}\left(\int_{\R^2} \log(1+|y|)(F(y,u)f(y,u))^{\frac{q}{q-1}}dy\right)^{(q-1)/q}\\
	\leq C(u)\|v\|_{1,q(w)}
	\end{multline*}
	thanks to Theorems \ref{thm_A1}, \ref{thm_Aq}, so that $\mathfrak F_1'(u) \in H^{-1}L^q_w$. The fact that $ \mathfrak F_1'(u_n) \to \mathfrak F_1'(u)$ in $H^{-1}L^q_w$ if $u_n \to u$ in $H^{1}L^q_w$ follows by similar arguments, we just sketch the proof:
	\begin{multline*}
	|\mathfrak F_1'(u_n)v-\mathfrak F_1'(u)v|\\
	=2\left|\int_{\R^2}\int_{\R^2}\log(1+|x-y|)\left[F(x,u_n(x))f(y,u_n(y))-Fx,(u(x))f(y,u(y))\right]v(y)\right| dx dy\\
	\leq 2\int_{\R^2}\int_{\R^2}\log(1+|x-y|)\left |F(x,u_n(x))-Fx,(u(x))\right| f(y,u_n(y))|v(y)| dx dy +\\
	+ 2\int_{\R^2}\int_{\R^2}\log(1+|x-y|)\left |f(y,u_n(y))-f(y,u(y))\right| F(x,u(x))|v(y)| dx dy\\
	\leq 2\int_{\R^2}\int_{\R^2} \log(1+|x|)|F(x,u_n)-F(x,u)| dx\int_{\R^2}f(y,u)|v| dy\\
	+ 2\int_{\R^2} |F(x,u_n)-F(x,u)| dx\int_{\R^2}\log(1+|y|)f(y,u)|v| dy+\\
	+2\int_{\R^2} \log(1+|x|)F(x,u) dx\int_{\R^2}|f(y,u_n)-f(y,u)|v| dy\\
	+ 2\int_{\R^2} F(x,u) dx\int_{\R^2}\log(1+|y|)|f(y,u_n)-f(y,u)||v| dy
	\end{multline*}
	We conclude by applying Holder's inequality and Theorems \ref{thm_A1} and \ref{thm_Aq}, together with Corollaries \ref{cor-wcontinuity} and \ref{cor-qwcontinuity}, which guarantee the integrals involved are continuously bounded, thus 
	$$
	|\mathfrak F_1'(u_n)v-\mathfrak F_1'(u)v|\leq C(u){\rm{o}}_n(1)\|v\|_{1,q(w)}
	$$
	where ${\rm{o}}_n(1)$ tends to $0$ together with $u_n \to u$ in $H^1L^q_w$, as $n\to \infty$.

\noindent Consider $\mathfrak F_2$ and let $\{u_n\}$ be a sequence in $H^1L^q_w(\R^2)$ converging to some $u$ and assume $u\neq 0$. Since $u_n\to u$ in $H^1(\R^2)$, by \cite{dOdSdMS} we have 
	\begin{equation}\label{ser_ine_1}
	\int_{\R^2}|F(x,u_n)-F(x,u)|^p dx \to 0, \qquad \hbox{ for any } p>0\ .
	\end{equation}
	Hence
	\begin{multline*}
	|\mathfrak F_2(u_n)-\mathfrak F_2(u)|\\
	\leq\int_{\R^2}\int_{\R^2}\log(1+|x-y|^{-1}) \left|F(x,u_n(x))F(y,u_n(y))-F(x,u(x))F(y,u(y))\right| dxdy\\
	\leq \int_{\R^2}\int_{\R^2}\log(1+|x-y|^{-1}) F(x,u_n(x))\left|F(y,u_n(y))-F(y,u(y))\right| dxdy+\\
	+\int_{\R^2}\log(1+|x-y|^{-1}) \left|F(x,u_n(x))-F(x,u(x))\right| F(y,u(y)) dxdy\\
	\leq \int_{\R^2}\int_{\R^2}\frac{1}{|x-y|} F(x,u_n(x))\left|F(y,u_n(y))-F(y,u(y))\right| dxdy+\\
	+	\int_{\R^2}\int_{\R^2}\frac{1}{|x-y|} F(y,u(y))\left|F(x,u_n(x))-F(x,u(x))\right| dxdy\ .
	\end{multline*}
	The last two integrals tend to $0$, thanks to \eqref{ser_ine_1}, the continuity of the functional $\int F(x,u)dx$ on $H^1$ and the HLS inequality. In the case $u=0$ the proof is similar.\\
	\noindent 	For any $u\in H^1L^q_w$ the G\^ateaux derivative of $\mathfrak F_2$ at $u\in H^1L^q_w$ is given by
$$
	\mathfrak F_2'(u)v= 2\int_{\R^2}\int_{\R^2}\log(1+|x-y|^{-1}) F(x,u(x))v(y)f(y,u(y))dxdy\ .
	$$
	Since
	\begin{equation*}
	|\mathfrak F_2'(u)v |\leq  2\int_{\R^2}\int_{\R^2}\frac{1}{|x-y|} F(x,u(x))v(y)f(y,u(y))
	\leq  C(u)\|v\|_2 \leq C(u)\|v\|_{1,q(w)}
	\end{equation*}
	thanks to \eqref{orlicz} and HLS inequality, one has $\mathfrak F_2'(u) \in H^{-1}L^q_w$. The fact that $ \mathfrak F_2'(u_n) \to \mathfrak F_2'(u)$ in $H^{-1}L^q_w$ if $u_n \to u$ in $H^{1}L^q_w$ is similar to previous cases.
	
\noindent Clearly from  $\mathfrak F=\mathfrak F_2-\mathfrak F_1$ one has $F\in \mathcal C^1$ on $H^1L^2_w$.
\end{proof}

\section{The variational framework}\label{MP_geo} 
\noindent As we are going to see, the variational framework for problem \eqref{P} is non-standard as we will exploit the de Figuereido--Miyagaki--Ruf type condition $(f_4)$ to prove estimates of the mountain pass level for the energy functional $I_V$. Let us first establish the Mountain Pass geometry in the next 
\begin{lem}\label{lem-MPG}
	The energy functional $I_V$ satisfies:
	\begin{enumerate}
		\item[(1)] there exist $\rho, \delta_0>0$ such that $I_V|_{S_{\rho}}\geq
		\delta_0>0$ for all $$u \in S_{\rho}=\{u\in H^1L^q_w: \|u\|_{1,q(w)}=\rho\};$$
		\item[(2)] there exists $e\in H^1L^q_w$ with $\|e\|_{1,q(w)}>\rho$ such that $I_V(e)<0$.
	\end{enumerate}
\end{lem}
\begin{proof}
From $\|u\|_V^2\leq \|u\|^2_{1,q(w)}$, if $\rho$ is small then the $H^1$ norm $\|u\|_V$ is also small. As a consequence of the logarithmic HLS inequality,
\begin{multline*}
 \mathfrak F(u)=\frac 1{4\pi}\int_{\R^2}\left[\log \frac 1{|x|}\ast F(x,u)\right]F(x,u) dx \\
 \leq \| F(x,u)\|_1\left[C\|F(x,u)\|_1+\int_{\R^2} F(x,u)\log F(x,u) dx-\| F(x,u)\|_1\log \| F(x,u)\|_1\right]\ .
\end{multline*}
By \eqref{caoineq} and recalling estimates \eqref{estF}, since the Sobolev norm is small enough we have that the $L^1$ norm of $F(u)$ can be bounded as follows
$$
\|F(u)\|_1\leq C\left(\|u\|^2_2+\|u\|_q^q\right) \leq C \|u\|^2_V\ .
$$
Moreover, on one eside 
\begin{multline*}
\int_{\R^2} F(x,u)\log F(x,u) dx\leq \int_{\R^2} |F(x,u)\log F(x,u)| dx\leq c_1\int_{\R^2}u^re^{4\pi^2}dx+c_2\int_{\R^2}u^2dx\\
\leq c_3\|u\|_{2r}^r\left(\int_{\R^2}e^{\alpha u^2}-1 dx\right)^{\frac 12}+c_2\int_{\R^2}u^2 dx\leq c\|u\|_V^2
\end{multline*}
for some $r>2$ and $\alpha>4\pi$. On the other side,
$$
|\| F(x,u)\|_1\log \| F(x,u)\|_1|\leq c |\|u\|_V^2\log \|u\|_V|\leq \|u\|_V
$$
provided the Sobolev norm is small and in turn we obtain
$$
\mathfrak F(u)\leq c\|u\|^3_V
$$
and thus 
$$
I_V(u)\geq \frac 12 \|u\|_V^2-c \|u\|_V^3 = \delta_0>0
$$
where $\delta_0$ depends only on $\rho$. This yields the first claim of the Lemma. 

\medskip

\noindent Let now $e$ be a smooth function, compactly supported in a small ball, say $B_{1/4}$. Then
\begin{multline*}
 \mathfrak F(e)=\frac 1{4\pi}\int_{\R^2}\left[\log \frac 1{|x|}\ast F(x,e)\right]F(x,e) dx\\
 = \frac 1{4\pi}\int_{\R^2}\int_{\R^2}\log \frac 1{|x-y|} F(x,e(x))F(y,e(y))dx dy \\
 \geq \frac{\log 2}{4\pi}\int_{\R^2}\int_{\R^2} F(x,e(x))F(y,e(y))dx dy = \frac{\log 2}{4\pi}\left(\int_{\R^2} F(x,e)dx\right)^2
\end{multline*}
since $F(x,e(x))F(y,e(y))\neq 0$ only for $|x|, |y|<1/4$, which implies $|x-y|<1/2$ and thus $|x-y|^{-1}>2$. Similarly,
$$
I_V(te)=\frac 12 t^2\|e\|_V^2-\mathfrak F(te)\leq \frac 12 t^2\|e\|_V^2-\frac{\log 2}{4\pi}\left(\int_{\R^2} F(x,te)dx\right)^2\to -\infty
$$
since $F$ has exponential growth, as $t\to +\infty$. 
\end{proof}

\noindent By the Ekeland Variational Principle \cite{E}, there exists a Palais--Smale
sequence (PS in the sequel) $\{u_n\}\subset  H^1L^q_{w}(\mathbb R^2)$ such that
$$
I'_{V}(u_n)\rightarrow 0,\quad I_{V}(u_n)\rightarrow {m_{V}},
$$
where the Mountain Pass level $m_{V}$ can be characterized by
\begin{equation} \label{m}
0<m_{V}:=\inf_{\gamma \in \Gamma} \max_{t\in [0,1]} I_{V}(\gamma
(t))
\end{equation}
where
$$
\Gamma:=\left\{\gamma \in \mathcal C^1([0,1], H^1L^q_{w})\ : \ \gamma(0)=0,
I_{V}(\gamma(1))<0\right\}.
$$
The next energy level estimate will be crucial in the sequel, in particular in proving compactness, for which $1/2$ turns out to be a substitute of the Sobolev level $(1/N)S^{N/2}$ in higher dimensions.
\begin{lem}\label{MPlevel-estimate}
	The mountain pass level $m_{V}$ satisfies
	$$
	m_{V}<\frac 12.
	$$
\end{lem}
\begin{proof}
	It is enough  to prove that there exists  a function $w\in H^1L_{q,w}$, with	$\|w\|_{1,q,w}=1$, such that
	$$
	\max_{t\geq 0}I_{V}(tw)<\frac 12.
	$$
	Let us introduce the following Moser type functions with support in
	$B_{\rho}$ 
	$$
	\overline{w}_n=\frac{1}{\sqrt{2\pi}}\left\{%
	\begin{array}{ll}
	\displaystyle    \sqrt{\log n}, & 0\leq |x|\leq \frac{\rho}n, \\
	\\
	\displaystyle    \frac{\log(\rho/|x|)}{\sqrt{\log n}}, & \frac{\rho}n\leq |x|\leq {\rho}, \\
	\\
	0,& |x|\geq {\rho}\ , \\
	\end{array}%
	\right.
	$$	where $\rho$ will be fixed later on.
	One has that
	\begin{multline*}
	\nonumber \|\overline
	w_n\|^2_{1,q(w)}=\int_{B_{\rho}}|\nabla\overline
	w_n|^2+ V(x)\overline w_n^2 dx+\left[\int_{B_{\rho}}\log(1+|x|)|\overline w_n^q dx\right]^{2/q}\\
\leq \int_{{\rho}/n}^{\rho}\frac{dr}{r\log
		n}+V_{\rho}\int_0^{{\rho}/n}\log n\,rdr+ (2\pi)^{\frac 2q-1}\log n\left[\int_0^{{\rho}/n}\log(1+r)\,rdr\right]^{2/q}\\
		+V_{\rho}\int_{{\rho}/n}^{\rho}\frac{\log^2(\rho/r)}{\log n}\,rdr+\frac{(2\pi)^{\frac 2q-1}}{\log n}\left[\int_{{\rho}/n}^{\rho}\log^q(\rho/r)\,\log(1+r)rdr\right]^{2/q}\\
\leq 1+V_{\rho}\frac{\rho^2}{2}\frac{\log n}{n^2}+(2\pi)^{\frac 2q-1}\log n\left[\int_0^{{\rho}/n}\log(1+r)\,rdr\right]^{2/q}\\
		+V_{\rho}\int_{{\rho}/n}^{\rho}\frac{\log^2(\rho/r)}{\log n}\,rdr+(2\pi)^{\frac 2q-1}\frac{\log^{2/q}(1+\rho)}{\log n}\left[\int_{{\rho}/n}^{\rho}\log^q(\rho/r)\,rdr\right]^{2/q}\\
	=1 +V_{\rho}\int_{{\rho}/n}^{\rho}\frac{\log^2(\rho/r)}{\log n}\,rdr+(2\pi)^{\frac 2q-1}\frac{\log^{2/q}(1+\rho)}{\log n}\left[\int_{{\rho}/n}^{\rho}\log^q(\rho/r)\,rdr\right]^{2/q}+{\text{o}}\left(\frac{1}{\log n}\right)\ .
	\end{multline*}
	The two integral terms in the previous expression can be estimated  as follows. On the one hand 
	\begin{equation*}
	\int \log^k(\rho/r)rdr=\frac{r^2}2\sum_{j=0}^k\left(\log(\rho/r)\right)^{k-j}\frac{k(k-1) \cdots  (k-j+1)}{2^j}
	\end{equation*}
	as well as 
	$$
	V_{\rho}\int_{{\rho}/n}^{\rho}\frac{\log^2(\rho/r)}{\log n}\,rdr= \frac{ \rho^2V_{\rho}}{4\log n}+ {\text{o}}\left(\frac{1}{\log n}\right)\ .
	$$
	On the other hand, since $q$ may be integer or not, a rough estimate reads as follows 	
    \begin{multline*}
    \int_{{\rho}/n}^{\rho}\log^q({\textstyle{\frac \rho r}})\,rdr\leq \int_{{\rho}/n}^{\rho}\left\{\log^{[q]}({\textstyle{\frac \rho r}})+\log^{[q]+1}({\textstyle{\frac \rho r}})\right\}\,rdr=\frac{ \rho^2[q]!}{2^{[q]+1}}\left[1+\frac{[q]+1}{2}\right]+ {\text{o}}\left(1\right)
	\end{multline*}
	so that eventually 
	 \begin{equation*}\label{delta_n1}
	1\leq  \|\overline w_n\|^2_{1,q(w)}
    \leq 1 +\delta_n + {\text{o}}\left(\frac 1{\log n}\right) ,
	  \end{equation*}
	 where 
		\begin{eqnarray}\label{delta_n2}
	\delta_n= \frac{\rho^2}{4\log n}\left[V_{\rho}+(2\pi)^{\frac 2q-1}\log^{\frac 2q}(1+\rho)\frac{[q]!}{2^{[q]-1}}\left(1+\frac{[q]+1}{2}\right)\right]+{\text{o}}\left(\frac{1}{\log n}\right)
	\end{eqnarray}
		Setting $w_n=\overline w_n/\sqrt{1+\delta_n}$, we get
	$\|w_n\|_{1, q(w)}\leq 1$.
	 We claim that there exists $n$ such that
	\begin{equation}\label{claim}
	\max_{t\geq 0}I_V(tw_n)<\frac 12.
	\end{equation}
	Let us argue by contradiction and suppose this is not the case, so
	that for all $n$ let $t_n>0$ be such that
	\begin{equation}\label{bycontr-assump}
	\max_{t\geq 0}I_{V}(tw_n)=I_{V}(t_nw_n)\geq \frac 12\ .
	\end{equation}
	 Then $t_n$ satisfies $$\frac{d}{dt}I_{V}(tw_n)|_{t=t_n}=0$$ and
	\begin{equation}\label{t_n^2=}
	t^2_n\geq  \frac 1{2\pi}\int_{\mathbb R^2}\left[\log \frac{1}{|x|}\ast
	F(x,t_nw_n) \right]t_nw_nf(x,t_nw_n) dx,
	\end{equation}
	\begin{equation}\label{t_n^2>}
	t^2_n\geq 1 + \frac 1{2\pi}\int_{\mathbb R^2}\left[\log
	\frac{1}{|x|}\ast F(x,t_nw_n)\right]F(x,t_nw_n) dx\ .
	\end{equation}
	Note that in \eqref{t_n^2=} we have an inequality instead of the equality since we know $\|w_n\|_{1, q,w}^2=1$, whereas in the energy functional it appears the equivalent, though smaller norm $\|w_n\|_V$. Actually the two norms differ for a quantity which is $O(1/\log n)$.
	
	\noindent From now on let us suppose
	$
	\rho \leq 1/2
	$. 
	This will simplify a few estimates, since for any $(x,y)\in {\text{supp }}w_n\times {\text{supp }}w_n$ we will have $|x-y|>1$, and in turn $\log(1/|x-y|)>0$. Let us now proceed in three steps:
	
	\medskip
	
	\noindent {\bf{Step 1.}} The following holds $\limsup_{n \to +\infty} t_n^2 \geq
	1$.
	
	\noindent Let us assume by contradiction that
	$\limsup_{n}t^2_n<1$: this implies that, up to a subsequence,
	there exists a positive constant $\delta_0$ such that $t_n^2\leq
	1-\delta_0$ for $n$ large enough. Since $\rho
	\leq \frac 12$, for any $|x|< \rho$, the set $\{y: |x-y|>1,
	|y|<\rho\}$ is empty. Recalling that the functions $w_n$ are
	compactly supported in $B_{\rho}$ we have
	\begin{multline*}
	\int_{\mathbb R^2}\left[\log \frac{1}{|x|}\ast
	F(x,t_nw_n)\right]F(x,t_nw_n) dx\\\
	=\int_{B_{\rho}}\int_{|x-y|\leq 1}\log
	\frac{1}{|x-y|} F(x,t_nw_n(x))F(y,t_nw_n(y))dxdy\geq 0
	\end{multline*}
	and thus a contradiction with \eqref{t_n^2>}.\\
	
	\medskip
	
	\noindent {\bf{Step 2.}} The following holds $\liminf_{n \to +\infty} t_n^2 \leq
	1$.
	
	\noindent Let us suppose by contradiction that $\liminf_{n \to +\infty}
	t_n^2>1$. Hence, up to a subsequence, there exists a constant
	$\delta_0>0$ such that
	$$
	t_n^2\geq 1+\delta_0
	$$
	as $n \to +\infty$.
	Let us estimate from below the right hand side of 
	\eqref{t_n^2=} (taking into account the possible negative sign of
	the logarithmic function):
	\begin{multline}\label{est1}
	\int_{\mathbb R^2}\left[\log \frac{1}{|x|}\ast
	F(x,t_nw_n)\right]t_nw_nf(x,t_nw_n) dx\\
	=\int_{|x|\leq \frac{\rho}n,|y|\leq
		\frac{\rho}n} \log
	\frac{1}{|x-y|}F(x,t_nw_n(x))t_nw_nf(y,t_nw_n(y))dxdy+\\
	+ \int_{\R^2 \times \R^2 \setminus \{|x|\leq \frac{\rho}n,|y|\leq
		\frac{\rho}n\} }\log
	\frac{1}{|x-y|}F(x,t_nw_n(x))t_nw_nf(y,t_nw_n(y))dxdy=I_1+I_2
	\end{multline}
	Thanks to $(f_4)$ we have for any $\varepsilon>0$ (here we choose $\varepsilon=\beta/2$),
	\begin{equation}\label{estimate-sfF}
	sf(s)F(s)\geq (\beta-\varepsilon)\cdot \frac 1{s^2}\cdot e^{8\pi s^2}=\frac{\beta}{2s^2}\cdot e^{8\pi s^2}, \quad \hbox{ for all
	}s\geq s_{\varepsilon}=s_{\beta}\ .
	\end{equation}
By the very definition of $w_n$ and since 
	$|x-y|< 2\rho/n<1$, we can estimate, for $n$ large enough, $I_1$ as follows
	\begin{multline*}
	I_1=\int_{B_{\rho/n}}t_nw_nf(y,t_nw_n)dy\int_{B_{\rho/n}}\log\frac{1}{|x-y|}F(x,t_nw_n)\,dx\\
	= \int_{B_{\rho/n}}t_n\frac{\sqrt{\log n}}{\sqrt{
			2\pi(1+\delta_n)}}f\left(y,t_n\frac{\sqrt{\log n}}{\sqrt{ 2\pi
			(1+\delta_n)}}\right)dy\\
			\cdot\,\int_{B_{\rho/n}}\log\frac{1}{|x-y|}
	F\left(x,t_n\frac{\sqrt{\log n}}{\sqrt{2\pi (1+\delta_n)}}\right)\,dx\\
	\geq 2\pi\beta\frac{e^{4 t_n^2 (1+\delta_n)^{-1} \log
			n}}{ 2t_n^2 (1+\delta_n)^{-1} \log n}\int_{B_{\rho/n}}dy\int_{B_{\rho/n}}\log\frac{1}{|x-y|}\,dx.
	\end{multline*}
	The last integral can be estimated as follows
	\begin{equation*}
	\int_{B_{\rho/n}}dy\int_{B_{\rho/n}}\log\frac{1}{|x-y|}dx
	\geq \int_{B_{\rho/n}}dy\int_{B_{\rho/n}}\log\frac{n}{|2\rho|}dx=
	\pi^2\left(\frac{\rho}{n}\right)^4 \log\frac{n}{2\rho}
	\end{equation*}
	As a consequence, we obtain
	\begin{equation}\label{est1.1}
	I_1\geq \pi^3\rho^4\beta\frac{e^{4 (t_n^2
			(1+\delta_n)^{-1}-1) \log n}}{t_n^2 (1+\delta_n)^{-1} \log
		n}\log\frac{n}{2\rho}\geq \pi^3\rho^4\beta t_n^{-2}e^{4
		(\frac{t_n^2}{1+\delta_n}-1) \log n}
	\end{equation}
	for any $n\geq n(\rho, \beta)$.
	Note that since $\rho \leq 1/2$ we have
	$$
I_2\geq 0 \ .
	$$
	Now, combining \eqref{t_n^2=}, \eqref{est1} and \eqref{est1.1} yields
	\begin{equation}\label{est1.2}
	t_n^4\geq \pi^3\rho^4\beta
	e^{4(\frac{t_n^2}{1+\delta_n}-1)\log  n}
	\end{equation}
	which is a contradiction, either  if $t_n\to +\infty$ or $t_n$ stays bounded with $t_n^2\geq
	1+\delta_0$. The proof of Step 2 is then completed. Observe that, as a consequence of Step 1 and Step 2
	$$
	t_n^2\to 1 \quad \hbox{as } n\to +\infty\ .
	$$
	Moreover, as a byproduct of \eqref{est1.2}, we also have for some $C>0$,
	$$
	e^{4 (\frac{t_n^2}{1+\delta_n}-1) \log n}\leq C
	$$
	that is
	\begin{equation}\label{est2}
	\frac{t_n^2}{1+\delta_n}\leq 1+\frac{C}{\log
		n}=1+{\rm{O}}\left(\frac{1}{\log n}\right)\ .
	\end{equation}

\medskip

\noindent {\bf{Step 3.}} We are now in the condition of getting a contradiction and determine the quantity $\mathcal V$ which appears in condition $(f_4)$. We have proved that $t_n^2\to 1$. Moreover, we also know that $t_n^2\geq 1$ by \eqref{t_n^2>}, since $\rho \leq 1/2 $. By \eqref{est1.2}, recalling definition \eqref{delta_n2} of  $\delta_n$, we have 
	\begin{multline*}
1+\text{o}(1)\geq t_n^4\geq \pi^3\rho^4\beta e^{4(\frac{t_n^2}{1+\delta_n}-1)\log  n} \geq \pi^3 \rho^4\beta e^{-4\frac{\delta_n}{1+\delta_n}\log n}\\ \geq \pi^3 \rho^4 \beta e^{-\rho^2\left(V_{\rho}+(2\pi)^{\frac 2q-1}\log^{\frac 2q}(1+\rho)\frac{[q]!}{2^{[q]-1}}\left(1+\frac{[q]+1}{2}\right)\right)+\text{o}(1)}
	\end{multline*}
	where $V_\rho=\max_{|x|\leq \rho} V(x)$. Passing to the limit, we obtain 
	\begin{equation}\label{contr}
1\geq \pi^3 \rho^4\beta e^{-\rho^2\left(V_{\rho}+(2\pi)^{\frac 2q-1}\log^{\frac 2q}(1+\rho)\frac{[q]!}{2^{[q]-1}}\left(1+\frac{[q]+1}{2}\right)\right)\ .}
	\end{equation}
	Now set in assumption $(f_4)$ 
	$$\mathcal V:=\inf_{|x|\leq 1/2}\frac 1{\pi^3} |x|^{-4}e^{|x|^2\left(V_{1/2}+(2\pi)^{\frac 2q-1}\log^{\frac 2q}(1+|x|)\frac{[q]!}{2^{[q]-1}}\left(1+\frac{[q]+1}{2}\right)\right)},$$ a quantity which is actually a minimum, since the right hand function is continuous and unbounded as $|x|\to 0$ . Finally, fix $\rho\in (0,1/2]$ such that
	$$
	\beta>\frac 1{\pi^3} \rho^{-4}e^{\rho^2\left(V_{1/2}+(2\pi)^{\frac 2q-1}\log^{\frac 2q}(1+\rho)\frac{[q]!}{2^{[q]-1}}\left(1+\frac{[q]+1}{2}\right)\right)} $$
	to get 
	$$ \pi^3 \rho^4\beta e^{-\rho^2\left(V_\rho+(2\pi)^{\frac 2q-1}\log^{\frac 2q}(1+\rho)\frac{[q]!}{2^{[q]-1}}\left(1+\frac{[q]+1}{2}\right)\right)}>1
	$$
		which contradicts \eqref{contr}.
\end{proof}

\section{Properties of Palais-Smale sequences}\label{PS_sec}
\noindent In this Section we prove that the weak limit in $H^1L^q_{w}(\R^2)$ of the PS sequence for $I_V$ given by the Ekeland Variational Principle, which we know from Section \ref{MP_geo} is at the energy level $m_V<1/2$, is actually a weak nontrivial solution of \eqref{P}. As we are going to see, the presence of the sign changing factor $\log(|x|)$ makes the
estimates rather delicate. We start with the following Lemma in which we prove boundedness of PS sequences at any level $c<1/2$.

\begin{lem}\label{lem-PSbdd}
Assume that $(V_1)$--$(V_2)$ and $(f_1)$--$(f_4)$ hold. Let
$\{u_n\}\subset H^1L^q_{w}$ be an arbitrary PS sequence for $I_V$ at level $c$, namely 
$$
I_V(u_n)\to c< \frac 12 \text{ and } I_V'(u_n)\to 0 \quad \hbox{in }
H^{-1}L^q	_{w}(\mathbb R^2), \quad \hbox{ as } n\to +\infty\ .
$$
Then, the sequence $u_n$ is bounded in $H^1(\R^2)$ as well as 
\begin{equation*}
	\left|\int_{\R^2}\left[\log \frac{1}{|x|}\ast F(x,u_n)\right]F(x,u_n) dx\right|\leq C, \qquad
	\left|\int_{\R^2}\left[\log \frac{1}{|x|}\ast F(x,u_n)\right]u_nf(x,u_n) dx\right|\leq C\ .
	\end{equation*}
\end{lem}
\begin{proof}
Let $\{u_n\}\in H^1L^q_{w}(\R^2)$ be a PS sequence for $I_V$, namely as $n\to\infty$
	is,
	\begin{equation}\label{convI}
	\frac 12 \|u_n\|_V^2-\frac 1{4 \pi} \int_{\mathbb
		R^2}\left[\log\left(\frac{1}{|x|}\right)\ast
	F(x,u_n)\right]F(x,u_n) dx \to c
	\end{equation}
	and
	\begin{multline}\label{convI'}
	\left|\int_{\mathbb R^2}\nabla u_n\nabla v+Vu_nv dx-\frac
	1{4\pi} \int_{\mathbb R^2}\left[\log\left(\frac{1}{|x|}\right)\ast
	F(x,u_n)\right]v f(x,u_n) dx\right.\\
	 \left.-\frac 1{4 \pi} \int_{\mathbb
		R^2}\left[\log\left(\frac{1}{|x|}\right)\ast v f(x,u_n)\right]F(x,u_n) dx\right|\\
= \left|\int_{\mathbb R^2}\nabla u_n\nabla v+Vu_nv dx - \frac 1{2
		\pi}\int_{\mathbb R^2}\left[\log\left(\frac{1}{|x|}\right)\ast
	F(x,u_n)\right]v f(x,u_n) dx \right|\leq \tau_n \|v\|_{H^{-1}L^q_{w}}
	\end{multline}
	for all $v\in H^{-1}L^q_{w}(\mathbb R^2)$, where $\tau_n\to 0$ as $n\to
	+\infty$. Since $H^1(\R^2)\hookrightarrow H^{-1}L^q_{w}(\R^2)$ and $H^1L^q_{w}(\R^2)\hookrightarrow H^1(\R^2)$, we can take $v=u_n$ in \eqref{convI'}, to obtain 
	\begin{equation}\label{I''}
	\left|\|u_n\|_V^2- \frac 1{2\pi}\int_{\mathbb
		R^2}\left[\log\left(\frac{1}{|x|}\right)\ast
	F(x,u_n)\right]u_nf(x,u_n)dx \right| \leq \tau_n \|u_n\|_{H^{-1}L^q_{w}}\leq C\tau_n \|u_n\|_{V}
	\end{equation}
	where we have also used the fact that $\|\cdot\|_V$ is an equivalent norm to the standard one in $H^1(\R^2)$. Another suitable choice for test function is given by 
	$$
	v_n=\frac{F(x,u_n)}{f(x,u_n)}=\frac{c(x)F(u_n)}{c(x)f(u_n)}=\frac{F(u_n)}{f(u_n)}
	$$
	Indeed, since  $f(s)=0$ if and only if $s=0$, by $(f_2)$ we have that $0\leq v_n\leq C u_n   $  (actually, it is uniformly bounded) so that $v_n$ is well defined and in $L^q(wdx)$. Furthermore, 
	\begin{equation*}
	\nabla v_n=\nabla u_n\frac{f^2(u_n)-F(u_n)f'(u_n)}{f^2(u_n)}=\nabla u_n\left(1-\frac{F(u_n)f'(u_n)}{f^2(u_n)}\right)\ .
	\end{equation*}
	Since the quantity $(F/f)'$  is bounded  by $(f_2)$, see also \eqref{F/f}, we have
	\begin{equation*}
	|\nabla v_n|^2\leq C|\nabla u_n|^2
	\end{equation*}
	so that $v_n\in H^1L^q_{w}(\R^2)$. Taking $v=v_n=\frac{F(u_n)}{f(u_n)}$ in \eqref{convI'} yields
	\begin{multline}\label{I'''}
	\left|\int_{\R^2}|\nabla u_n|^2\left(1-\frac{F(u_n)f'(u_n)}{f^2(u_n)}\right) dx+\int_{\R^2}V(x)u_n\frac{F(x,u_n)}{f(x,u_n)} dx\right.\\
	\left.-\frac{1}{2\pi} \int_{\R^2} \left[\log \frac{1}{|x|}\ast F(x,u_n)\right]F(x,u_n)\right| \\ 
	\leq \tau_n\| v_n\|_{H^{-1}L_{q,w}} \leq \tau_n \|u_n\|_V\ .
	\end{multline}
	\noindent Now recall \eqref{convI}, namely 
	\begin{equation*}
	\frac 1{2 \pi} \int_{\mathbb
		R^2}\left[\log\left(\frac{1}{|x|}\right)\ast
	F(x,u_n)\right]F(x,u_n) dx=\|u_n\|_V^2-2c+\text{o}(1)\ .
	\end{equation*}
	Only two cases may occur as $n\to +\infty$ (we are not excluding that both the two cases may appear for different subsequence of $u_n$):
	\begin{itemize}
		\item $\int_{\mathbb R^2}\left[\log\left(\frac{1}{|x|}\right)\ast
		F(x,u_n)\right]F(x,u_n) dx\leq 0$: in this case we have, directly, 
		$$
		\|u_n\|_V^2\leq 2c+ \text{o}(1)\leq 3c
		$$
		and $c\geq 0$;
		\item $\int_{\mathbb R^2}\left[\log\left(\frac{1}{|x|}\right)\ast
		F(x,u_n)\right]F(x,u_n)> 0$. In this case, combining \eqref{convI} and \eqref{I'''} yields
		\begin{multline*}
		\int_{\R^2}\left(|\nabla u_n|^2+Vu^2_n\right) dx-2c+\text{o}(1)\\
		-\int_{\R^2}|\nabla u_n|^2\left(1-\frac{F(u_n)f'(u_n)}{f^2(u_n)}\right) dx-\int_{\R^2}V(x)u_n\frac{F(u_n)}{f(u_n)} dx\leq \tau_n \|u_n\|_V
		\end{multline*}
	\end{itemize}	
		\noindent Note that, as a consequence of $(f_2)$, see also \eqref{F/f}, $F(x,s)\leq (1-\delta)sf(x,s)$. Thus we have, as $n\to\infty$
		\begin{equation*}
		\delta \int_{\R^2}|\nabla u_n|^2+Vu_n^2 dx\leq \tau_n \|u_n\|_V +2c+\text{o}(1)\ .
		\end{equation*}

	\noindent In conclusion we have proved that
	$$
	\|u_n\|_V\leq C\ .
	$$
	As a consequence, from \eqref{convI} and \eqref{I''}, we also have 
	\begin{equation*}
	\left|\int_{\R^2}\left[\log \frac{1}{|x|}\ast F(x,u_n)\right]F(x,u_n) dx\right|\leq C, \qquad
	\left|\int_{\R^2}\left[\log \frac{1}{|x|}\ast F(x,u_n)\right]u_nf(x,u_n) dx\right|\leq C,
	\end{equation*}
	that is our thesis.
 
\end{proof}

\noindent Differently from standard contexts in which having proved boundedness of a PS sequence brings the conclusion at hand, here it does not allow to employ standard arguments to prove the weak limit is actually a nontrivial solution to the equation. Indeed, the presence of the exponential nonlinearity together with the sign-changing behavior of the logarithmic kernel, prevents the application of standard estimates. Here comes into play the key estimate for the mountain pass level $m_V<1/2$ established in Lemma \ref{MPlevel-estimate}.
\begin{lem}\label{lem-PSimproveTM}
Assume $(V_1)-(V_2)$ and $(f_1)-(f_4)$. Let
$\{u_n\}\subset H^1L^q_{w}$ be a PS sequence for $I_V$ at level $c<1/2$. Then,  for any $1\leq \alpha <\frac{1}{2c}$ the following uniform bound holds  
\begin{equation*}
\sup_{n\in \mathbb N}\int_{\R^2}[F(x,|u_n|)]^\alpha<\infty\ .
\end{equation*}
\end{lem}
\begin{proof}
By Lemma \ref{lem-PSbdd} the sequence $\{u_n\}$ is bounded in  $H^1(\R^2)$ and we may assume $u_n\rightharpoonup u$ in $H^1(\R^2)$, $u_n\to u$ in $L^s_{loc}(\R^2)$ for any $1\leq s< \infty$ and  $u_n\to u$ a.e. in $\R^2$, with
	\begin{equation*}
	\lim_{n\to +\infty} \|u_n\|_V^2=A^2\geq \|u\|_V^2
	\end{equation*}
	As in the proof of Lemma \ref{lem-PSbdd}, we will carefully select a suitable test function $v_n$. Let us introduce the following auxiliary function
$$
G(t)=\int_0^t \frac{\sqrt{F(s)f'(s)}}{f(s)}ds
$$
which is well defined and $\mathcal C^1$ thanks to $(f_2)$. Moreover, by H\"{o}lder's inequality we have 
\begin{multline}\label{estG}
G^2(t)\leq \int_0^tds \cdot \int_0^t \frac{Ff'}{f^2}ds=t\left[\int_0^t \left(\frac{Ff'-f^2}{f^2}+1\right)ds\right]=\\=
t\left[\int_0^t-\frac{d}{ds}\frac{F}{f}+t\right]=t^2-t\frac{F(t)}{f(t)}
\end{multline}
Define 
$$v_n:=G(u_n)\ ,$$
then
\begin{equation*}
\int_{\R^2}|\nabla v_n|^2 dx=\int_{\R^2}|\nabla u_n|^2\frac{Ff'}{f^2}(u_n) dx\leq C, \quad \int_{\R^2}Vv^2_n dx=\int_{\R^2}VG^2(|u_n|) dx\leq C
\end{equation*}
as $u_n$ is bounded in $H^1$ and applying again $(f_2)$. We aim at proving that 
$$\|\nabla v_n\|_2^2+\|\sqrt{V(x)}v_n\|_2^2\leq 1$$ as $n$ is large enough. First, note that as $n\to +\infty$, one has $$0\leq \int_{\R^2}VG^2(|u_n|) dx\leq  \int_{\R^2}Vu_n^2 dx\leq C\ .$$ 

\noindent In order to estimate the norm $\|v_n\|_V^2$ recall  \eqref{convI} and  \eqref{I'''}. From $\| u_n\|_V^2\to A^2\geq \| u\|_V^2$, we have
\begin{equation*}
\lim_{n\to +\infty}\frac 1{2 \pi} \int_{\mathbb
	R^2}\left[\log\left(\frac{1}{|x|}\right)\ast
F(x,u_n)\right]F(x,u_n) dx =A^2- 2c
\end{equation*}
and 
\begin{multline*}
\left|\int_{\R^2}|\nabla u_n|^2\left(1-\frac{F(u_n)f'(u_n)}{f^2(u_n)}\right) dx + \int_{\R^2}V(x)u_n\frac{F(u_n)}{f(u_n)} dx \right.\\
\left.-\frac{1}{2\pi} \int_{\R^2} \left[\log \frac{1}{|x|}\ast F(x,u_n)\right]F(x,u_n) dx \right|\to 0
\end{multline*}
so that
\begin{multline*}
\int_{\R^2}|\nabla u_n|^2\left(1-\frac{F(u_n)f'(u_n)}{f^2(u_n)}\right) dx+\int_{\R^2}V(x)u_n\frac{F(u_n)}{f(u_n)} dx+2c\\
 - \int_{\R^2}|\nabla u_n|^2 dx -\int_{\R^2}V u_n^2 dx = {\rm{o}}(1)
\end{multline*} and in turn
\begin{multline}\label{normvn}
\|v_n\|_V^2=\int_{R^2}|\nabla G(u_n)|^2 dx +\int_{\R^2}V(x)G^2(u_n) dx \\
=2c+\int_{R^2} V(x)\left(u_n\frac{F(u_n)}{f(u_n)}-u_n^2+G^2(u_n)\right) dx+{\rm{o}}(1)\leq 2c +{\rm{o}}(1) <1
\end{multline}
by \eqref{estG}, as $n$ is large enough. 

\noindent Once we have estimated the norm of $v_n$, let us take advantage of this to improve the exponential integrability of the original sequence $u_n$. By $(f_3)$, for any $\epsilon>0$ there exists a constant $t_\epsilon>0$ such that
$$
1-\epsilon<\frac{\sqrt{F(t)f'(t)}}{f(t)}\leq 1+\epsilon, \quad \text{ for all }\ t\geq t_\epsilon\ .
$$
Next by  $(f_2)$ we also have either $u_n(x)\leq t_\epsilon $ or $u_n(x)\geq t_\epsilon$ which implies 
\begin{multline}\label{vn-un}
v_n\geq \int_{0}^{t_\epsilon}\delta dt + \int_{t_\epsilon}^{u_n}(1-\epsilon)dt\geq \delta t_\epsilon +(1-\epsilon)(u_n-t_\epsilon)\geq (1-\epsilon)(u_n-t_\epsilon)
\end{multline}
and thus 
$$
u_n\leq t_\epsilon + \frac{v_n}{1-\epsilon}, \quad \hbox{ for any } x\in \mathbb R^2\ .
$$
Hence (hereafter $C_\epsilon$ may change from line to line)
\begin{multline}\label{Falpha1}
\int_{\R^2}\left[ F(x,u_n)\right]^\alpha dx = \int_{u_n\leq t_\epsilon}\left[ F(x,u_n)\right]^\alpha dx + \int_{u_n\geq t_\epsilon}\left[ F(x,u_n)\right]^\alpha dx\\
\leq C_\epsilon\int_{u_n\leq t_\epsilon}\left[ u^2_n\right]^\alpha dx +\int_{u_n\geq t_\epsilon}\left[ F\left(x,t_\epsilon + \frac{v_n}{1-\epsilon}\right)\right]^\alpha dx\\
\leq C_\epsilon \int_{u_n\leq t_\epsilon}u^2_n dx +C\int_{u_n\geq t_\epsilon} \left(t_\epsilon + \frac{v_n}{1-\epsilon}\right)^{\alpha(p-1)}e^{4\pi \alpha (t_\epsilon + \frac{v_n}{1-\epsilon})^2} dx\\ 
\leq  C_\epsilon \|u_n\|_2^2+C_\epsilon\int_{u_n\geq t_\epsilon} e^{4\pi \alpha(1+\epsilon)(t_\epsilon + \frac{v_n}{1-\epsilon})^2}dx
\end{multline}
where, in the last line, we use the following inequality: for any $T>0$ and for any $\epsilon >0$ there exists $C=C_{T, \epsilon}$ such that $s^{p-1}\leq C_{T, \epsilon} e^{4\pi \epsilon s^2}$ for any $s\geq T$ (with $T= t_\epsilon$ already fixed as well as $C=C_\epsilon$). Moreover, for any $\epsilon>0$ there exists $C_\epsilon$ such that $(t+s)^2\leq C_\epsilon t^2+ (1+\epsilon)s^2$ for any $s,t>0$. Then
\begin{equation}\label{ser_est_2}
\left(t_\epsilon + \frac{v_n}{1-\epsilon}\right)^2 \leq C_\epsilon t^2_\epsilon+ (1+\epsilon)\left(\frac{v_n}{1-\epsilon}\right)^2\ .
\end{equation}
As byproduct of \eqref{vn-un}, if $u_n\geq t_\epsilon$ then $v_n\geq \delta t_\epsilon$. Combining this with \eqref{ser_est_2} and \eqref{Falpha1} we obtain
\begin{multline*}
\int_{\R^2}[F(x,u_n)]^\alpha dx\leq  C_\epsilon \|u_n\|_2^2+C_\epsilon\int_{u_n\geq t_\epsilon} e^{4\pi \alpha(1+\epsilon)^2\frac{v^2_n}{(1-\epsilon)^2}} dx \\
\leq 
C_\epsilon \|u_n\|_2^2+C_\epsilon\int_{\R^2} e^{4\pi \alpha(1+\epsilon)^2\frac{v^2_n}{(1-\epsilon)^2}}-1 dx\ .
\end{multline*}
Let us now fix $0<\epsilon<1$ and set 
$$
\eta:= \frac 1{2c}-\alpha >0, \quad \epsilon_\alpha:= c^2\eta^2 = c^2 \left(\frac 1{2c}-\alpha\right)^2<\frac 14\ .
$$
With these choices we obtain
$$
\int_{\R^2}[F(x,u_n)]^\alpha dx\leq C_\alpha \|u_n\|_2^2+C_\alpha\int_{\R^2} e^{4\pi \alpha\frac{(1+\epsilon_\alpha)^2}{(1-\epsilon_\alpha)^2}\|v_n\|_V^2\frac{v^2_n}{\|v_n\|_V^2}}-1 dx\ .
$$
By \eqref{normvn}, $\|v_n\|_V^2\leq 2c+\rm{o}(1)$ as $n$ is large enough, so that
$$
\|v_n\|_V^2\leq 2c+4c^2\eta,  \quad \hbox{ as } n\to +\infty\ .
$$
Hence,
$$
\alpha\frac{(1+\epsilon_\alpha)^2}{(1-\epsilon_\alpha)^2}\|v_n\|_V^2\leq 2c \left(\frac 1{2c}-\eta\right)\frac{(1+c^2\eta^2)^2}{(1-c^2\eta^2)^2}(1+2c\eta)=\frac{(1+c^2\eta^2)^2}{(1-c^2\eta^2)^2}(1-4c^2\eta^2)<1
$$
since the last inequality is equivalent to
$$
(1+c^2\eta^2)^2(1-4c^2\eta^2)=(1+c^4\eta^4+2c^2\eta^2)(1-4c^2\eta^2)<(1-c^2\eta^2)^2= 1+c^4\eta^4-2c^2\eta^2\ .
$$
In conclusion we have 
$$
\int_{\R^2}[F(x,u_n)]^\alpha dx\leq C_\alpha \|u_n\|_2^2+C_\alpha\int_{\R^2} e^{4\pi \frac{v^2_n}{\|v_n\|_V^2}}-1 dx \leq C_{\alpha}
$$
by the Ruf inequality \eqref{Ri} and Remark \ref{Ruf_rem}.
\end{proof}
\begin{Prop}\label{prop-PS}
	Assume that conditions $(V_1)$--$(V_2)$ and $(f_1)$--$(f_4)$ are satisfied. Let
	$\{u_n\}\subset H^1L^q_{w}$ be a PS sequence for $I_V$ at level $c<1/2$, weakly converging to $u$ in $H^1$. If $u\neq 0$, then $u\in H^1L^q_{w}$ and 
	$u_n\rightharpoonup u$ weakly in $H^1L^q_{w}$. Furthermore, as $n\to\infty$
	\begin{eqnarray} \label{convFF}
	\left[\log |x| \ast F(x,u_n)\right]f(x,u_n)\longrightarrow\left[\log
	|x|\ast F(x,u)\right]f(x,u) \quad \hbox{ in } L_{loc}^1(\mathbb R^2)
	\end{eqnarray}
	and $u$ is a weak solution to \eqref{P}.
\end{Prop}
 \begin{proof}  Fix $\alpha \in (1, 1/2c)$ so that by \eqref{convI} we have
	\begin{multline*}
	\frac 12 \|u_n\|^2_V+\int_{\R^2}\int_{\R^2}\log\left(1+|x-y|\right)F(x,u_n(x))F(y,u_n(y)) dxdy=\\=c+\int_{\R^2}\int_{\R^2}\log\left(1+\frac{1}{|x-y|}\right)F(x,u_n(x))F(y,u_n(y))dxdy+ {\text{o}}(1)\\
	\leq c+1+C_\alpha\int_{\R^2}\int_{\R^2}\left[1+\frac{1}{|x-y|^{4\frac{\alpha -1}{\alpha}}}\right] F(x,u_n(x))F(y,u_n(y))dxdy \\ 
	\leq c+1+C_\alpha \left\{\int_{\R^2}[F(x,u_n)]^\alpha dx \right\}^2+C_\alpha \left\{\int_{\R^2}F(x,u_n) dx\right\}^2
	\end{multline*}
by Proposition \ref{HLS}, since $\frac 2{\alpha}+4\frac{\alpha -1}{\alpha}\, \frac 12=2$. Hence, by Lemma \ref{lem-PSimproveTM}, 
$$
\int_{\R^2}\int_{\R^2}\log\left(1+|x-y|\right)F(x,u_n(x))F(y,u_n(y))dxdy\leq C_\alpha
$$
as $n\to +\infty$, and thus also $\int_{\R^2\times \R^2}\log\left(1+|x-y|\right)|u_n(x)|^q|u_n(y)|^qdxdy$ is bounded. Since $u\neq 0$, by Lemma 2.1 in \cite{CW} we have
$$
\int_{\R^2}\log(1+|x|^2)|u_n|^q(x)dx\leq C \quad \hbox{ as } n \to +\infty
$$
so that $\|u_n\|^2_{1,q(w)}$ is bounded. Up to a subsequence we have $u_n\rightharpoonup u$ in $H^1L^q_{w}$. Moreover, recall that as $n\to +\infty$,
\begin{equation}\label{Lim}
\frac{1}{2\pi}\int_{\R^2}\left[\log \frac{1}{|x|}\ast F(x,u_n)\right]vf(x,u_n) dx= \int_{\R^2}\nabla u \nabla v+ V(x)uv dx + {\rm{o}}(1), \quad \, v \in H^{-1}L^q_{w}, 
\end{equation}
in particular for any $\varphi \in \mathcal C^{\infty}_c(\R^2)$. In order to prove that  $u$ is a weak solution of \eqref{P}, let us suppose for the moment the following 

\noindent {\textbf{Claim:}}
\begin{equation*}
\int_{\R^2} \left|\log \frac{1}{|x|}\ast F(x,u_n)\right|f(x,u_n)|u_n| dx\leq C
\end{equation*}
of which we postpone the proof. 
Now we apply Lemma 2.1 in \cite{dFMR} to the sequence of functions 
$$g(y, u_n(y)):=\left(\log \frac{1}{|x|}\ast F(x,u_n)\right)(y)f(y,u_n(y)),$$ restricted to any compact domain $\Omega$: they are $L^1$ functions since $u_n, u\in H^1L^q_{w}$ and, thanks to the claim, $u_n(y)g(y, u_n(y))$ is uniformly bounded in $L^1$. Therefore, from \cite{dFMR} we have 
\begin{equation*}
\left(\log \frac{1}{|x|}\ast F(x,u_n)\right)f(x,u_n) dx \to \left(\log \frac{1}{|x|}\ast F(x,u)\right)f(x,u) dx\ \ \hbox{ in } L^1_{loc}(\R^2)
\end{equation*}
as well as  
\begin{equation*}
\int_{\R^2}\left[\log \frac{1}{|x|}\ast F(x,u_n)\right] f(x,u_n) \varphi dx\to \int_{\R^2}\left[\log \frac{1}{|x|}\ast F(x,u)\right]f(x,u)\varphi dx 
\end{equation*}
for any $\varphi \in \mathcal C^{\infty}_c(\R^2)$, which is a dense subset of $H^{-1}L^q_{w}$. This together with \eqref{Lim} implies that $u$ is a weak solution of \eqref{P}.

\noindent \textit{Proof of the Claim.}

\noindent The key ingredient is the uniform bound provided by Lemma \ref{lem-PSimproveTM}. In order to simplify the notation, let us set
\begin{equation*}
w_n(y)=\left(\log\frac 1{|x|}\ast F(x,u_n)\right) (y)
\end{equation*}
By \eqref{I''} 
$$\int_{\R^2}w_nf(x,u_n)u_n dx= A^2+o(1)\ ,
$$
where $A=\lim_{n\to +\infty}\| u_n\|_V\geq \|u\|_V$ so that  
$$\int_{\R^2}w_nf(x,u_n)u_n dx >0$$ 
for $n$ large enough (note that we are assuming $u\neq 0$, that is, $A^2>0$). Hence,
\begin{equation*}
0<\int_{\R^2}w_nf(x,u_n)u_n dx=\int_{w_n>0}w_nf(x,u_n)u_n dx +\int_{w_n<0}w_nf(x,u_n)u_n dx
\end{equation*}
which implies 
$$\int_{\R^2}w^-_nf(x,u_n)u_n dx<\int_{\R^2}w^+_nf(x,u_n)u_n dx$$
and thus 
\begin{equation*}
\int_{\R^2}|w_n|f(x,u_n)u_n dx\leq 2\int_{\R^2}w^+_nf(x,u_n)u_n dx\ .
\end{equation*}
Now, we have 
\begin{multline*}
\int_{\R^2}w^+_nf(x,u_n)u_n=\int_{y:w_n>0}f(y,u_n(y))u_n(y)dy\int_{\R^2}\log(1+\frac 1{|x-y|}) F(x,u_n(x))dx\\
-\int_{y:w_n>0}f(y,u_n(y))u_n(y)dy\int_{\R^2}\log(1+|x-y|) F(x,u_n(x))dx\\
\leq \int_{\R^2}\int_{\R^2}\log(1+\frac 1{|x-y|}) F(x,u_n(x))f(y,u_n(y))u_n(y)dxdy\ .
\end{multline*}
Therefore, for any $\mu>0$, small, there exists a constant $\delta_\mu \in (0,1)$ sufficiently small such that
\begin{multline*}
\int_{\R^2}w^+_nf(x,u_n)u_n dx\leq \int_{\R^2}dy\int_{|x-y|>\delta_\mu}\log(1+\delta_\mu^{-1}) F(x,u_n(x))f(y,u_n(y))u_n(y)dx+\\+2\int_{\R^2}dy\int_{|x-y|<\delta_\mu}\frac 1{|x-y|^{\mu}} F(x,u_n(x))f(y,u_n(y))u_n(y) dx\\
\leq  C_{\mu}\int_{\R^2} F(x,u_n(x))dx \int_{\R^2}f(y,u_n)u_n(y)dy \\
+ \int_{\R^2}\int_{\R^2}\frac{1}{|x-y|^{\mu}} F(x,u_n(x))f(y,u_n)(y)u_n(y)dxdy\ .
\end{multline*}
The first integral in the last expression is uniformly bounded, as one can see by Lemma \ref{lem-PSimproveTM} and Holder's inequality, recalling that $\|u_n\|_V$ is also uniformly bounded.  Concerning the second term, by the HLS inequality if $2/s +\mu/2=1$, one has 
\begin{equation*}
\int_{\R^2}\int_{\R^2}\frac{1}{|x-y|^{\mu}} F(x,u_n(x))f(y,u_n(y))u_n(y)dxdy\leq C(\mu)\|F(x,u_n)\|_s\|f(y,u_n)u_n\|_s\ .
\end{equation*}
Since
\begin{equation*}
s=\frac{4}{4-\mu}\to 1 \ \ \hbox{ as } \ \mu \to 0 
\end{equation*}
we can  choose  $\mu$ small enough to apply again  Lemma \ref{lem-PSimproveTM} and  Holder's inequality, to obtain that $\|F(x,u_n)\|_s\|f(x,u_n)u_n\|_s$ stays bounded. Finally, $\int w^+_nf(x,u_n)u_n$ is bounded and the same holds for $\int |w_n|f(x,u_n)u_n$, that is our claim.
	\end{proof}

	\section{Proof of Theorem \ref{thm1}}\label{final_sec}
	
	\noindent We are now ready to prove Theorem \ref{thm1}. From Lemma \ref{lem-MPG}, the functional $I_V$ satisfies the Mountain Pass geometry. Hence, there exists a (PS) sequence  $\{u_n\}\subset H^1L^q_{w}(\R^2)$ at level ${m_V}$ and by Lemma \ref{lem-PSimproveTM}, $\{u_n\}$ is bounded in $H^1$ and it weakly converges to some $u\in H^1$. We have that either $\{u_n\}$ is vanishing, that is for any
$r>0$
$$
\lim_{n\to +\infty}\sup_{y\in \mathbb R^2}\int_{B_r(y)}|u_n|^2 dx=0
$$
or there exist $r, \delta >0$ and a sequence $\{y_n\}\subset \mathbb Z^2$ such that
$$
\lim_{n\to \infty}\int_{B_r(y_n)}|u_n|^2 dx\geq \delta.
$$
If $\{u_n\}$ is vanishing, by Lions' concentration-compactness result we have
\begin{equation}\label{Lions}
u_n\to 0 \quad \mbox{in} \quad  L^s(\mathbb R^2) \quad \forall \, s>2,
\end{equation}
as $n\to\infty$. In this case it is standard to show that
$$
\|F(x,u_n)\|_{\gamma}, \|u_nf(x,u_n)\|_{\gamma}\to 0
$$
for some values of $\gamma>1$ and close to 1, thanks to the improved exponential integrability given by Lemma \ref{lem-PSimproveTM} and the growth assumption $F(x,t)<tf(x,t)$. Hence, applying the HLS inequality we deduce, similarly to the conclusion of the proof of Proposition \ref{prop-PS}:
\begin{eqnarray}
\int_{\R^2}\int_{\R^2} \log\left(1+\frac{1}{|x-y|}\right) F(x,u_n(x))F(y,u_n(y))dxdy\to 0\label{lab_lim_1}\\ 
\int_{\R^2}\int_{\R^2} \log\left(1+\frac{1}{|x-y|}\right) F(x,u_n(x))u_n(y)f(y,u_n(y)) ddx dy\label{lab_lim_2}\to 0
\end{eqnarray}
as $n\to\infty$.  Combining \eqref{lab_lim_1}--\eqref{lab_lim_2} with \eqref{convI} and \eqref{Lim} yields
$$
\frac 1{2\pi}\int_{\R^2}\int_{\R^2} \log\left(1+|x-y|\right) F(x,u_n(x))\left[F(y,u_n(y))-u_n(y)f(y,u_n(y))\right] dx dy= 2m_V+{\text{o}}(1)
$$
so that $m_v\leq 0$, which is not possible. Therefore the vanishing case does not occur.\\

\noindent Now set $v_n:=u_n(\cdot-y_n)$, then
\begin{equation}\label{nonvan}
\int_{B_r(0)}|v_n|^2 dx\geq \delta\ .
\end{equation}
By the periodicity assumption, $I_V$ and $I'_V$ are both invariant by the $\mathbb Z^2$-action, so that $\{v_n\}$ is
still a PS sequence at level $m_V$. Then $v_n\rightharpoonup v$ in $H^1(\R^2)$ with $v\neq 0$ by using \eqref{nonvan}, since $v_n\to v $ in $L^2_{loc}(\mathbb R^2)$. We conclude by Proposition \ref{prop-PS} that $v\in H^1L^q_{w}$ is a nontrivial critical point of $I_V$ and $I_V(v)=m_V$, which completes the proof of Theorem \ref{thm1}.

\end{document}